\documentclass[11pt,final]{amsart}
\usepackage{amsmath,amssymb,amsthm,amsfonts,mathrsfs,amsopn}
\usepackage[all]{xy}
\usepackage{dsfont}
\usepackage{hyperref}
\usepackage{color}
\usepackage{esint}
\usepackage{mathtools}
\mathtoolsset{showonlyrefs}
\usepackage{slashed}

\usepackage{tikz-cd}

\usepackage{faktor}

\newcommand{\tnorm}[1]{%
  \left\vert\kern-0.3ex\left\vert\kern-0.3ex\left\vert #1 \right\vert\kern-0.3ex\right\vert\kern-0.3ex\right\vert
}

\usepackage[margin=0.9in]{geometry}

\usepackage[obeyDraft]{todonotes}

\newcommand{\p}{\partial}
\newcommand{\R}{\mathbb{R}}
\newcommand{\C}{\mathbb{C}}
\newcommand{\T}{\mathbb{T}}

\newcommand{\D}{\slashed{D}}

\newcommand{\pd}{\slashed{\partial}}
\newcommand{\dd}{\mathop{}\!\mathrm{d}}
\newcommand{\sph}{\mathbb{S}}

\newcommand{\eps}{\varepsilon}

\newcommand{\pD}{\mathbf{D}}

\newcommand{\Abracket}[1]{\left<#1\right>} 
\newcommand{\parenthesis}[1]{\left(#1\right)} 
\newcommand{\braces}[1]{\left\{#1\right\}} 

\newcommand{\bZ}{\mathbb{Z}}

\DeclareMathOperator{\dist}{dist}
\DeclareMathOperator{\Dom}{Dom}
\DeclareMathOperator{\dv}{\dd{vol}}

\DeclareMathOperator{\Eigen}{Eigen}

\DeclareMathOperator{\id}{Id}

\DeclareMathOperator{\Mtp}{Multi} 

\DeclareMathOperator{\proj}{Pr}

\DeclareMathOperator{\Span}{Span}

\DeclareMathOperator{\Spect}{Spect}

\DeclareMathOperator{\vol}{vol}

\DeclareMathOperator{\up}{up}
\DeclareMathOperator{\down}{down}

\newtheorem{thm}{Theorem}[section]

\newtheorem{lemma}[thm]{Lemma}
\newtheorem{prop}[thm]{Proposition}

\newtheorem{rmk}[thm]{Remark}

\title[Splitting Periodic Dirac waves]{Stationary periodic solutions for nonlinear Dirac equations with non-coercive nonlinearity II: splitting}


\thanks{The authors are supported by Beijing Natural Science Foundation No. 1262018.}


\author[R. Wu]{Ruijun Wu}
\address{Ruijun Wu, School of Mathematics and Statistics, Beijing Institute of Technology, Zhongguancun South Street No. 5, 100081 Beijing, P.R. China.}
\email{ruijun.wu@bit.edu.cn}

\author[F. Zhang]{Fuping Zhang}
\address{Fuping Zhang, School of Mathematics and Statistics, Beijing Institute of Technology, Zhongguancun South Street No. 5, 100081 Beijing, P.R. China.}
\email{fuping.zhang@bit.edu.cn}

\begin{document}

\begin{abstract}
    We study stationary periodic solutions of nonlinear Dirac equations with Soler-type nonlinearities. 
    By splitting off a circle factor the problem is reduced to dimension two.
    The nonlinearity degenerates along a Lorentz null cone, causing difficulties for the variational analysis. 
    Using a coercive perturbation we first obtain minimax perturbed solutions. 
    The key new ingredient is a quantitative separation of the lowest positive eigenspace of the reduced Dirac operator from the null cone; together with a uniform resolvent estimate, it yields uniform bounds that allow us to remove the perturbation.
    This produces nontrivial periodic solutions for temporal frequencies near the corresponding positive spectral threshold, including frequencies above the mass.

\end{abstract}

\maketitle

\section{Introduction}

Motivated by physical models for fermionic particles, we consider the stationary nonlinear Dirac equation
\begin{align}\label{eq:NDE-spatial}
    -\sum_{j=1}^3 i\gamma^0\gamma^j\p_j\psi + m\gamma^0\psi =a\psi + \p F
\end{align}
on a flat three-dimensional torus~$\T^3$, where~$m>0$ stands for the mass,~$a>0$ is a temporal frequency,~$F\colon \T^3\times \C^4\to \R$ is the nonlinearity standing for the self-coupling interaction and~$\p F(\psi(x))\equiv \p_\psi F(x,\psi(x))$ denotes the Nemytskii operator associated to~$F$.
The~$\gamma^\mu$'s are given as follows.
The matrix~$\gamma^0$ takes the form
\begin{align}
    \gamma^0= \begin{pmatrix} I_2 & 0 \\ 0 & -I_2\end{pmatrix}
    =\begin{pmatrix}
        1 & & &  \\ & 1 & & \\ & & -1 & \\ & &  & -1
    \end{pmatrix}
\end{align}
and induces a decomposition of the spinors into~$\pm1$ eigenspaces of~$\gamma^0$.
As for the others, we denote by
\begin{align}
    {\sigma^1} = \left( {\begin{array}{*{20}{c}}
			{0}&1 \\
			1&{0}
	\end{array}} \right),\;{\sigma^2} = \left( {\begin{array}{*{20}{c}}
			{0}&{ - i} \\
			i&{0}
	\end{array}} \right),\;{\sigma^3} = \left( {\begin{array}{*{20}{c}}
			1&{0} \\
			{0}&{ - 1}
	\end{array}} \right)
\end{align}
for the Pauli matrices, and the gamma matrices are
\begin{align}
    {\gamma^k} = \left( {\begin{array}{*{20}{c}}
			{0}&{{\sigma^k}} \\
			{{-\sigma^k}}&{0}
	\end{array}} \right),\quad\text{for}\quad k= 1, 2, 3.
\end{align}
Note that they satisfy the following Clifford relations: for~$j,k\in\braces{1,2,3}$,
\begin{align}
    \sigma^j\sigma^k+\sigma^k\sigma^j= 2\delta^{jk}I_2, & &
    \gamma^j\gamma^k + \gamma^k\gamma^j= -2\delta^{jk}I_4.
\end{align}
Moreover,
\begin{align}
    \gamma^0\gamma^k + \gamma^k \gamma^0 =0, & &
    (\gamma^0)^2=I_4.
\end{align}
Then~$\sum_{\mu} i\gamma^\mu \p_\mu$ is the four-dimensional space-time Dirac operator.

Equations of such forms have been widely studied in physics, chemistry and mathematics, see~\cite{Esteban2002overview,Ranada1983Classical,Thaller1992Dirac} and the references therein. 
They are governed by a general simple Dirac equation in the Minkowski spacetime 
\begin{align}\label{eq:NDE-spacetime}
    \sum_{\mu=0}^3 i\gamma^\mu\p_\mu \Psi - m \Psi +\gamma^0 \p_\Psi F(\cdot,\Psi)=0, \qquad \mbox{ in } \; \R^{1+3}. 
\end{align}
Here~$x^0$ stands for the time coordinate, and the equation is hyperbolic. 
The equation can be reduced to~\eqref{eq:NDE-spatial} if we consider stationary solutions, namely spinors of the form
\begin{align}\label{eq:ansatz}
    \Psi(x^0, \vec{x}) = e^{-iax^0}\psi(\vec{x})
\end{align}
for some~$a\in \R\setminus\braces{0}$. 
Note that~$|\Psi|^2=|\psi|^2$, so~$\Psi$ and~$\psi$ possess the same probability density, and as time evolves the state is essentially unchanged in shape. 
Additional hypotheses on~$F$ are required before one can really treat~\eqref{eq:NDE-spacetime} as an equation on~$\R^3$, for example~$F$ should not depend on the time coordinate explicitly, and~$F$ is invariant under phase rotation:
\begin{align}
    F(x, e^{i\theta}\psi)=F(x,\psi), \quad \forall e^{i\theta} \in U(1). 
\end{align}
 
Now we restrict to consider~\eqref{eq:NDE-spatial}.
A large part of the variational literature deals with nonlinearities of Ambrosetti--Rabinowitz type, namely, the nonlinearities consist of functions in the Euclidean amplitude~$|\psi|$ and small perturbations. 
In that setting the main difficulties arise from the strong indefiniteness of the functional, failure of Palais--Smale conditions, and the competition with the nonlinear term, see for example~\cite{BartschDing2006Solutions,DingRuf2008Solutions,YangDing2012Stationary}.
On general compact spin manifolds, related variational problems have also been studied, see for instance~\cite{Ammann2003Habil, BartschXu2021spinorial, DingLiXu2016bifurcation, DingXu2025localized, Isobe2011nonlinear, Isobe2020Morse}.

Another class of nonlinearities consists of those involving Lorentz type scalars such as~$\bar{\psi}\psi\equiv \Abracket{\gamma^0\psi,\psi}$. 
They arise naturally in standard fermionic models.
In~\cite{Soler1970Classical} and the even earlier~\cite{Finkelstein1951Non-linea}, the authors considered the nonlinearity~$ (\bar{\psi}\psi)^2$ in the fermionic spinor functional. 
For this reason we refer to the nonlinearities of the form~$|\bar{\psi}\psi|^\nu$ as Soler-type nonlinearities; sometimes we also allow for a positive smooth function multiplied to this nonlinearity, standing for certain nonconstant field, as well as small perturbations added, see the examples in~\cite{WZ2026Stationary}. 

Writing the spinor as~$\psi=(\psi_{\up},\psi_{\down})$ as in~\eqref{eq:up-down}, we see clearly the degeneracy of this Lorentz scalar:
\begin{align}
    \bar{\psi}\psi=\Abracket{\gamma^0\psi,\psi}
    =|\psi_{\up}|^2 - |\psi_{\down}|^2.
\end{align}
Thus~$\bar{\psi}\psi$ may stay small even if~$|\psi|$ is large. 
Such a degeneracy is inherited by the nonlinearity~$F$: 
in the Lorentz null cone
\begin{align}
    \mathcal{N}=\braces{\psi\in\C^4 \mid \bar{\psi}\psi=0}, 
\end{align}
the nonlinearity~$F$ may be small, but this cone contains spinors of arbitrary large norm. 
This causes severe difficulties for the existence theory of~\eqref{eq:NDE-spatial}.
The autonomous case in~$\R^3$ has been studied extensively.
Via a radial ansatz, one can reduce to a system of ODEs in two functions, which can be solved by the shooting method as well as techniques from dynamical systems~\cite{BalabaneCazenave1988Existence,Cazenave1986Existence, EstebanSere1995stationary,Merle1988Existence}.

For symmetry reasons it is also natural to consider periodic solutions. 
In~\cite{DingLiu2014periodic} the authors dealt with nonlinearities of Ambrosetti--Rabinowitz type. 
Remark that there are also works studying~\eqref{eq:NDE-spatial} in~$\R^3$ which is non-autonomous but with~$F$ being periodic in the~$x$ variables~\cite{BartschDing2006Solutions,DingLiu2015Periodic,DingLiu2017Periodic}. 
For nonlinearities of Soler type, we recently obtained existence results for~$F$ satisfying additional structural constraints in~\cite{WZ2026Stationary}. 
The purpose of the present paper is to enlarge the frequency regime which admits nonzero solutions.

The main idea is to exploit the product decomposition~$\T^3=\T^2\times\sph^1$.
For a fixed circle mode~$\zeta$, the splitting ansatz
\begin{align}
    \psi=\varphi\otimes e^{i\zeta x^3}
\end{align}
reduces the equation to a nonlinear Dirac equation on~$\T^2$ with the lowest positive eigenvalue of the reduced Dirac operator being~$\sqrt{\zeta^2+m^2}$.
Although the corresponding eigenspace is no longer contained in the~$(+1)$-eigenspace of~$\gamma^0$, it still intersects the Lorentz null cone only at the origin. 
This leads to a quantitative null-cone separation principle. 
Together with a resolvent estimate on the orthogonal complement, this rules out blow-up of the perturbed minimax solutions whose levels tend to zero and whose normalized profiles approach the null cone. 
On the variational side, the perturbed functionals admit linking levels with upper bound~$C^*(a)$ such that 
\begin{align}
    C^*(a)\to 0 \qquad\mbox{as }a\to \sqrt{\zeta^2+m^2}. 
\end{align}
These two ingredients allow us to obtain estimates uniform in both the
coercive perturbation and~$a$ for~$a$ sufficiently close to~$\sqrt{\zeta^2+m^2}$
from below. 
Then we can pass to the zero-perturbation limit and obtain a nonzero solution. 

\ 

We need structure conditions on~$F\colon \T^3\times \C^4\to \R_+$ similar to those in~\cite{WZ2026Stationary} and in~\cite{EstebanSere1995stationary}, which are listed as follows. 
First of all, as already commented, we assume that~$F$ is invariant under phase rotation and does not depend on the third coordinate~$x^3$ in~$\T^3$.
Hence we can view~$F$ as defined on~$\T^2\times\C^4$. 
In addition, the following structural assumptions are imposed: there exist constants~$A_j>0$,~$1\leq j\leq 5$, and~$\nu>1$,~$2<\alpha_1\leq\alpha_2<4$,~$\alpha\in(0,1)$,~$\beta>2$ such that 
\begin{itemize}
    \item[(F1)] $ 0 \leq F(x,\varphi) \leq A_1 \big( |\varphi|^{\alpha_1} + |\varphi|^{\alpha_2} \big)$, $\forall x\in\T^2$ and $\forall \varphi \in \mathbb{C}^4 $;

    \item[(F2)] \( F \in C^{1,\alpha}_{loc} \), $F(x,0) =0 $, $\p_\psi F(x,0)=0$, and
          \( |\p_{\varphi} F(x,\varphi)| \leq A_2 |\varphi|^{\alpha_2 - 1} \) for \( |\varphi| \) large;

    \item[(F3)] \(  \p_\varphi F(x,\varphi)[\varphi]\geq \alpha_1 F(x,\varphi) \);

    \item[(F4)] \( F(x,\varphi) \geq A_3 |\bar{\varphi}\varphi|^\nu - A_4 \);

    \item[(F5)] for any~$\delta > 0$, there is a~$C_\delta > 0$, such that $\forall \varphi \in \mathbb{C}^4$,~$\forall x\in\mathbb{T}^2$,
        \begin{align}
        |\p_\varphi F(x,\varphi)| 
        \leq A_5 \left( \delta + C_\delta F(x,\varphi)^{\frac{1}{\beta}} \right) |\varphi|.
        \end{align}
\end{itemize}

\begin{rmk}
    Some more comments are in order. 
    \begin{itemize}
        \item[(i)]  By (F1) and (F4) we necessarily have~$\nu\leq \frac{\alpha_2}{2}<2$.

        \item[(ii)] The growth condition in (F2) holds for~$|\varphi|$ large. 
                    But together with (F1) we see that
                    \begin{align}\label{eq:F2-all}
                        |\p F(x,\varphi)|\leq C(1+|\varphi|^{\alpha_2-1}), \qquad \forall (x,\varphi)\in\T^2\times\C^4,  
                    \end{align}
                    for some constant~$C$. 
                    For later convenience we will assume~$A_2$ is large enough that~\eqref{eq:F2-all} holds with~$A_2$. 

        \item[(iii)] Examples of such nonlinearities satisfying (F1-5) include 
                    \begin{align}
                        F(x,\varphi)=h_1(x)|\bar{\varphi}\varphi|^{\frac{\alpha_1}{2}} + h_2(x)|\bar{\varphi}\varphi|^{\frac{\alpha_2}{2}}
                    \end{align}
                    where~$h_1,h_2\colon \T^2 \to \R_+$ are smooth uniformly bounded positive functions, and~$2<\alpha_1 \leq \alpha_2 <4$. 
                    Then we can take a suitable~$\beta\in (2,\frac{\alpha_2}{\alpha_2-2})$. 
                    For more examples we refer to~\cite{WZ2026Stationary}. 
    \end{itemize}
\end{rmk}

We take a lattice~$\Gamma(\ell_1,\ell_2,\ell_3)$ of rank three, and consider the tori 
\begin{align}
    \T^3 =\faktor{\R^3}{\Gamma(\ell_1,\ell_2,\ell_3)}, &  & 
    \T^2 =\faktor{\R^2}{\Gamma(\ell_1,\ell_2)}. 
\end{align}
We refer to Section~\ref{sect:splitting} for explanation of notations. 
\begin{thm}\label{thm:existence-splitting}
    Let~$\zeta\in\bZ\frac{2\pi}{\ell_3}$, ~$\theta=1-\frac{m^2}{36(\zeta^2+m^2)}$, and assume~$F\colon \T^2\times \C^4\to\R_+$ is phase-rotation invariant and satisfies (F1-5). 
    
    There exists~$m_*\in (\theta \sqrt{\zeta^2+m^2}, \sqrt{\zeta^2+m^2})$ such that for each~$a\in(m_*,\sqrt{\zeta^2+m^2})$, the equation~\eqref{eq:NDE-spatial} admits a nonzero~$C^1$ solution~$\psi$ on~$\T^3=\faktor{\R^3}{\Gamma(\ell_1,\ell_2,\ell_3)}$ of the form~$\psi=\varphi\otimes e^{i\zeta x^3}$.
\end{thm}
Note that for~$\zeta\neq 0$, the above solution is automatically nonconstant, at least in~$x^3$. 
Meanwhile, since~$\sqrt{\zeta^2+m^2}>m$, we get existence with~$a$ above the mass threshold~$m$, and as~$\zeta$ varies in~$\bZ\frac{2\pi}{\ell_3}$, we get larger frequency range than that in~\cite{WZ2026Stationary}. 

Our strategy is inspired by the radial ansatz reduction and more closely by~\cite{SireXu2023variational} where the authors directly use the circle factor to reduce the spinorial Yamabe equation to a codimension-one submanifold, turning the critical nonlinearity to a subcritical one. 
For technical reasons we can only deal with subquartic nonlinearities, but this covers the sub-cubic case in~\cite{WZ2026Stationary}. 

\ 

The paper is organized as follows. 
In Section~\ref{sect:splitting} we split out a circle factor and relate the spinors to those on the two-dimensional torus~$\T^2$.
For each~$\zeta$ in the spectrum of the circle that is split out, we reduce~\eqref{eq:NDE-spatial} to~\eqref{eq:NDE-T2} on~$\T^2$. 
The two-dimensional domain allows for nonlinearities of subquartic growth. 
In preparation for later variational analysis, we clarify the eigenvalues and eigenspaces of the Dirac type operators on~$\T^2$ in Section~\ref{sect:spectral}, and make two quantitative estimates concerning spinors away from the lowest positive eigenspace in Section~\ref{sect:key lemma}. 
These generalize the corresponding conclusions in~\cite{WZ2026Stationary}. 
We realize solutions of the reduced equation~\eqref{eq:NDE-T2} as critical points of a functional~$J_\zeta\colon H^{\frac{1}{2}}(\T^2,\C^4)\to\R$ and use a perturbation strategy which dates back to~\cite{EstebanSere1995stationary} and is well developed in~\cite{WZ2026Stationary}.
When the frequency~$a$ is in a left neighborhood of the minimal positive eigenvalue~$\sqrt{m^2+\zeta^2}$ of the reduced Dirac operator, the perturbed functional admits a local level-linking uniformly in~$\eps$.
Using the admissible deformation theory developed in~\cite{WZ2026Stationary}, the local linking geometry leads to a positive minimax level~$\Lambda_1(\zeta,\eps)$ for each~$\eps\in (0,1]$, which can be verified to be a critical level of~$J_{\zeta,\eps}$.
Moreover, we can estimate these critical levels in a fine way. 
At the end of Section~\ref{sect:perturbed solutions} we get a uniform~$H^1$ estimate of such perturbed minimax solutions under the condition that~$a$ is in a sufficiently small left neighborhood of~$\sqrt{m^2+\zeta^2}$. 
Passing to a convergent subsequence, we show that the perturbations in both the equation and the functional vanish in the limit~$\eps_n\to 0^+$. 
The limit is the desired nonzero solution in the main Theorem~\ref{thm:existence-splitting}. 
The regularity can also be verified since the zero order terms are H\"older continuous and the Dirac operator is elliptic of first order.

\section{Splitting of the torus and reduction of the equation}\label{sect:splitting}

In this section we clarify the relations between the spinors on~$\T^3$ and~$\T^2=\faktor{\T^3}{\sph^1}$, as well as the Dirac operators on these manifolds. 

Let~$\Gamma(\ell_1,\ell_2,\ell_3)=\ell_1\mathbb{Z} \times \ell_2\mathbb{Z} \times \ell_3\mathbb{Z}$ be a product lattice of rank three in~$\R^3$, whose dual lattice is 
\begin{align}
    \Gamma(\ell_1,\ell_2,\ell_3)^*=\frac{1}{\ell_1}\bZ\times \frac{1}{\ell_2}\bZ \times \frac{1}{\ell_3}\bZ.
\end{align}
The periodic functions and spinors can be viewed as defined on the torus
\begin{align}
    \T^3\coloneqq \faktor{\R^3}{\Gamma(\ell_1,\ell_2,\ell_3)} 
    = \sph^1\parenthesis{\frac{\ell_1}{2\pi}} 
    \times \sph^1\parenthesis{\frac{\ell_2}{2\pi}}
    \times \sph^1\parenthesis{\frac{\ell_3}{2\pi}}.
\end{align}
Since this torus admits circle actions (in more than one way), we can split out a circle factor and get 
\begin{align}\label{eq:product-str-mfld}
    \T^3= \T^2\times \sph^1(\frac{\ell_3}{2\pi})
\end{align}
where~$\T^2=\faktor{\R^2}{\Gamma(\ell_1,\ell_2)}$ and~$\Gamma(\ell_1,\ell_2)=\ell_1\bZ\times \ell_2\bZ\subset \R^2$ is a rank two lattice in~$\R^2$.
With a little abuse of notation we will write~$\T^3=\T^2\times\sph^1$. 

In this work we work with spinors which are~$\C^4$ valued functions on~$\T^3$.
This amounts to considering the trivial spin structure on~$\T^3$ and a double of the trivial spinor bundle~$\Sigma \T^3 \oplus\Sigma \T^3=(\T^3\times \C^2)\oplus (\T^3\times\C^2)=\T^3\times \C^4$.
Note that with respect to the product structure in~\eqref{eq:product-str-mfld}, we have
\begin{align}\label{eq:splitting-spinor-bundles}
    \pi_1^*\Sigma\T^2\otimes\pi_2^*\Sigma\sph^1\cong\Sigma\T^3=\T^3\times\C^2,
\end{align}
where
\begin{itemize}
    \item $\pi_1\colon\T^3\to\T^2$ and~$\pi_2\colon\T^3\to\sph^1$ are the projections,
    \item $\Sigma\T^2=\T^2\times \C^2$ is the trivial spinor bundle associated to the trivial spin structure of~$\T^2$,
    \item and~$\Sigma\sph^1=\sph^1\times\C$ is also the trivial spinor bundle over~$\sph^1$.
\end{itemize}
We refer to~\cite{Ginoux2009Dirac} and~\cite{Lawson1989Spin} for more information on spin geometry.

\begin{rmk}
    There are two spin structures on the circle.
    Take the circle~$\sph^1(\frac{\ell}{2\pi})$ of perimeter~$\ell$ for example.
    The two spin structures are:
    \begin{itemize}
        \item the trivial spin structure~$\xi_0$, with spinors represented by~$\ell$-periodic complex-valued functions;
        \item the twisted spin structure~$\xi_1$, with spinors represented by~$\ell$-anti-periodic complex-valued functions. 
    \end{itemize}
    Let~$\theta\in [0,\ell]$ denote the angular coordinate for~$\sph^1(\frac{\ell}{2\pi})$, then the Dirac operator, in both cases, takes the form 
    \begin{align}
        \pd_{\sph^1}=-i\frac{\dd}{\dd\theta}.
    \end{align}
    However, the spectra for the Dirac operators with respect to different spin structures are different. 
    In the trivial case,
    \begin{align}
    \Spect(-i\frac{\dd}{\dd \theta};\xi_0)= \bZ \frac{2\pi}{\ell}, 
    \end{align}
    and for each~$\zeta\in\mathbb{Z}\frac{2\pi}{\ell}$, the eigenspace is generated by~$e^{i\zeta\theta}$. 
    Meanwhile for the twisted case, 
    \begin{align}
    \Spect(-i\frac{\dd}{\dd\theta};\xi_1)=(\frac{1}{2}+\mathbb{Z})\frac{2\pi}{\ell}
    \end{align}
    and for each~$\zeta\in(\frac{1}{2}+\mathbb{Z})\frac{2\pi}{\ell}$, the eigenspace for~$\zeta$ is again generated by~$e^{i\zeta\theta}$. 
    We refer to the note~\cite{Varilly2006Dirac} for further information. 
    Note that in either case, for any~$\zeta\in\Spect(-i\frac{\dd}{\dd \theta})$, the corresponding eigenspace is generated by~$e^{i\zeta\theta}$, and the eigenspinors have \emph{constant length}. 
\end{rmk}

We now clarify the spinors in this setting. 
The spinors on~$\R^{1+3}$ are~$\C^4$-valued functions. 
By the stationary ansatz, we are essentially dealing with spinors on~$\R^3$, via the relation 
\begin{align}
    \Sigma\R^{1+3}|_{\braces{t}\times \R^3}\cong \Sigma\R^3 \oplus \Sigma \R^3. 
\end{align}
Accordingly, we consider spinors in the bundle~$\Sigma\T^3\oplus\Sigma\T^3$, which corresponds to the up-down formulation of spinors as in~\cite{WZ2026Stationary}. 
Indeed, by writing
\begin{align}\label{eq:up-down}
    \psi= \begin{pmatrix}\psi^1\\ \psi^2 \\ \psi^3 \\ \psi^4\end{pmatrix}
    \equiv \begin{pmatrix} \psi_{\up} \\ \psi_{\down} \end{pmatrix}, 
\end{align}
the components~$\psi_{\up},\psi_{\down}$ are spinors on~$\Sigma \T^3$.
By~\eqref{eq:splitting-spinor-bundles}, we can further regard~$\psi_{\up},\psi_{\down}$ as a section of~$\pi_1^*\Sigma\T^2\otimes\pi_2^*\Sigma\sph^1$.
This motivates the following consideration: take~$\psi$ of the form
\begin{align}\label{eq:spinor-splitting}
    \psi= \begin{pmatrix}
        \psi_{\up} \\ \psi_{\down}
    \end{pmatrix}
    = 
    \begin{pmatrix}
        \varphi_{\up}\otimes \phi_{\up} \\ 
        \varphi_{\down}\otimes \phi_{\down}
    \end{pmatrix}
\end{align}
where~$\varphi_{\up},\varphi_{\down}$ are sections of~$\Sigma\T^2$ and~$\phi_{\up},\phi_{\down}$ are sections of~$\Sigma\sph^1$.
We hope to find nice candidates of~$\phi_{\up},\phi_{\down}$ such that spinors in the form of~\eqref{eq:spinor-splitting} can solve~\eqref{eq:NDE-spatial}.

\ 

As in~\cite{WZ2026Stationary} we write the massless and massive Dirac operators as 
\begin{align}
    \D=\sum_{k=1}^3 -i\gamma^0\gamma^k \p_k, & & \mbox{ respectively } & & 
    \pD=\sum_{k=1}^3 -i\gamma^0\gamma^k\p_k+m\gamma^0 =\D+m\gamma^0, 
\end{align}
and the intrinsic Dirac operator on~$\mathbb{T}^3$ equipped with the trivial spin structure is denoted as 
\begin{align}
    \pd_{\T^3}=\sum_{k=1}^3 -i\sigma^k \p_k = -i\sigma\cdot \nabla. 
\end{align}
Then the equation~\eqref{eq:NDE-spatial} acting on~$\psi$ of the form~\eqref{eq:spinor-splitting} is equivalent to 
\begin{align}
    \pd_{\T^3}\psi_{\down} & + m\psi_{\up} - a \psi_{\up} - (\p F)_{\up}=0,  \\
    \pd_{\T^3}\psi_{\up} & -m\psi_{\down} - a \psi_{\down} - (\p F)_{\down }=0. 
\end{align}
The intrinsic Dirac operator acts on tensor product spinors in the following manner: 
\begin{align}
    \pd_{\T^3} \psi_{\up} 
    =& \sum_{k=1}^3  -i\sigma^k\p_k (\varphi_{\up}\otimes \phi_{\up}) \\
    =& \sum_{k=1}^2 -i\sigma^k(\p_k\varphi_{\up})\otimes \phi_{\up}  + (-i\sigma^3)\varphi_{\up}\otimes \p_3\phi_{\up} \\
    =& \parenthesis{\sum_{k=1}^2 -i\sigma^k\p_k\varphi_{\up} }\otimes \phi_{\up}
       + \sigma^3\varphi_{\up}\otimes \parenthesis{-i\frac{\dd}{\dd x^3}\phi_{\up}} \\
    =& (\pd_{\T^2}\varphi_{\up})\otimes\phi_{\up}
        +\sigma^3\varphi_{\up}\otimes (\pd_{\sph^1}\phi_{\up})
\end{align}
and similarly 
\begin{align}
    \pd_{\T^3}\psi_{\down}
    = \pd_{\T^3} (\varphi_{\down}\otimes \phi_{\down}) 
    = (\pd_{\T^2} \varphi_{\down})\otimes \phi_{\down}
     + \sigma^3\varphi_{\down}\otimes (\pd_{\sph^1}\phi_{\down}). 
\end{align}
\begin{rmk}
    In geometric literature one finds the following convention:
    \begin{align}
        \pd_{\T^3} \psi_{\up} 
    = \pd_{\T^3} (\varphi_{\up}\otimes \phi_{\up}) 
    = \pd_{\T^2} \varphi_{\up}\otimes \phi_{\up}
     + \omega_{\T^2}^{\C}\cdot\varphi_{\up}\otimes \pd_{\sph^1}\phi_{\up},
    \end{align}
    where~$\omega_{\T^2}^{\C}$ stands for the complex volume element of~$\T^2$:
    \begin{align}
    \omega_{\T^2}^{\C}=i(-i\sigma^1)(-i\sigma^2).
    \end{align}
    By calculating explicitly using our choice of the Pauli matrices, we find that 
    \begin{align}
        \omega_{\T^2}^{\C}= \begin{pmatrix} 1 & 0 \\ 0 & -1 \end{pmatrix} = \sigma^3.
    \end{align}
    Thus the two formulas are consistent. 
\end{rmk}
Consequently, we can reformulate the equation~\eqref{eq:NDE-spatial} as a system for~$\varphi_{\up},\varphi_{\down}$ provided the other factors~$\phi_{\up},\phi_{\down}$ can be chosen appropriately. 
On the other hand, for the Lorentz scalar~$\bar{\psi}\psi$, which now takes the form 
\begin{align}
    \bar{\psi}\psi=\Abracket{\gamma^0\psi,\psi}
    =|\psi_{\up}|^2 - |\psi_{\down}|^2
    =|\varphi_{\up}|^2 |\phi_{\up}|^2
     - |\varphi_{\down}|^2 |\phi_{\down}|^2.
\end{align}
This suggests choosing both circle factors to have constant lengths. 

Here is the observation of splitting method: if we take~$\phi_{\up}=\phi_{\down}=e^{i \zeta x^3}$ with~$\zeta\in\mathbb{Z}\frac{2\pi}{\ell_3}$ which satisfies 
\begin{align}
    \pd_{\sph^1} e^{i\zeta x^3}= \zeta e^{i\zeta x^3},
\end{align}
then the spinor~$\psi$ of the form~\eqref{eq:spinor-splitting} satisfies 
\begin{align}
    \bar{\psi}\psi=\bar{\varphi}\varphi, 
    \qquad \mbox{ with } 
    \qquad 
    \varphi 
    = \begin{pmatrix}
    \varphi_{\up} \\ \varphi_{\down}
    \end{pmatrix}, 
    \quad \mbox{ and } \quad 
    \psi=\varphi\otimes e^{i\zeta x^3
    }.
\end{align}
If the nonlinearity~$F$ is also independent of the~$x^3$ variable, then 
\begin{align}
    \p F(\psi) 
    =& \p_\psi F((x^1,x^2), \psi(x)) 
    = \p_\psi F((x^1,x^2),\varphi(x^1,x^2) e^{i\zeta x^3} )\\
    =& e^{i\zeta x^3}\p_\psi F(x^1,x^2,\varphi(x^1,x^2))
    = e^{i\zeta x^3}\p F(\varphi)
\end{align}
since we have assumed that~$F$ is invariant under phase rotation from the very beginning. 
In this case the new spinor~$\varphi\in \Gamma(\Sigma \T^2\oplus \Sigma \T^2)$ satisfies the equation 
\begin{align}
    \begin{pmatrix}
        0 & \pd_{\T^2} + \zeta \omega_{\T^2}^{\C} \\ \pd_{\T^2}+ \zeta \omega_{\T^2}^{\C} & 0
    \end{pmatrix}\varphi 
    + m\gamma^0\varphi - a\varphi -\p F(\varphi)=0, 
\end{align}

Thus we can reduce~\eqref{eq:NDE-spatial} to a system in~$\varphi\in\Gamma(\Sigma\T^2\oplus \Sigma\T^2)$. 
We now rewrite this system in a convenient form. 
Note that 
\begin{align}
    \begin{pmatrix} 0 & \omega_{\T^2}^{\C} \\ \omega_{\T^2}^{\C} & 0 \end{pmatrix} 
    =
    \begin{pmatrix} 0 & \sigma^3 \\ \sigma^3 & 0 \end{pmatrix}
    = \gamma^0 \gamma^3  = -\gamma^3\gamma^0.
\end{align}
Thus if we fix a~$\zeta$ and denote
\begin{align}\label{eq:Dirac-operators}
    \D_{\T^2}
    \coloneqq \begin{pmatrix}  0 & \pd_{\T^2} \\ \pd_{\T^2} & 0 \end{pmatrix}= -i\sum_{k=1}^2 \gamma^0\gamma^k\p_k, 
    & & \mbox{ and } & & 
    \pD_{\T^2} \coloneqq \D_{\T^2}+\zeta\gamma^0\gamma^3+m\gamma^0, 
\end{align}
then the equation in~$\varphi$ takes the form
\begin{align}\label{eq:NDE-T2}
    \pD_{\T^2}\varphi- a\varphi -\p F(\varphi)=0. 
\end{align}
This is an elliptic system on the two-dimensional torus~$\T^2= \faktor{\R^2}{\Gamma(\ell_1,\ell_2)}$.

\section{Spectral properties of the Dirac operators}\label{sect:spectral}

On a general closed spin manifold~$M$, the qualitative spectral properties of the operators 
\begin{align}
    \pd_M + f(x)\omega^{\C}_M, \qquad \mbox{ with } f\in C^\infty(M, \R_+)
\end{align}
are investigated in~\cite{SireXu2023variational}. 
Here we need more quantitative spectral properties. 
In our setting the coefficients are constant, and the torus Dirac spectra can be computed explicitly. 
Observe from the Clifford relations that
    \begin{align}
        \D_{\T^2}\parenthesis{\zeta\gamma^0\gamma^3+m\gamma^0}
        = -\parenthesis{\zeta\gamma^0\gamma^3+m\gamma^0} \D_{\T^2}, 
        & & 
        \mbox{ and } 
        & & 
        \parenthesis{\zeta\gamma^0\gamma^3+m\gamma^0}^2
        =\parenthesis{\zeta^2+m^2}I_4.
    \end{align}
Thus we have
\begin{align}
    \pD_{\T^2}^2
    =\parenthesis{ \D_{\T^2} + \zeta \gamma^0\gamma^3 + m\gamma^0}^2 
    =\D_{\T^2}^2 + \parenthesis{\zeta^2+m^2}I_4.
\end{align}
where~$I_4$ denotes the~$4\times 4$ identity matrix.
We first determine the spectrum of massless Dirac operator~$\D_{\T^2}$. 

Recall that the eigenvalues of~$\pd_{\T^2}$ associated to the trivial spin structure on the torus~$\T^2=\faktor{\R^2}{\Gamma(\ell_1,\ell_2)}$ are computed in~\cite{Friedrich1984zur} and~\cite{Ginoux2009Dirac}:  
\begin{align}
    \Spect(\pd_{\T^2}) =
    \braces{ \pm 2\pi |\zeta^*| \mid \zeta^*\in\Gamma(\ell_1,\ell_2)^*},
\end{align}
and the complex multiplicity of~$0$ is~$2$ (with eigenspinors given by the two linearly independent constant spinors, say~$(1,0)^T$ and~$(0,1)^T$), while the complex multiplicity of the nonzero eigenvalue~$2\pi|\zeta^*|$ equals the number of elements in~$\Gamma(\ell_1,\ell_2)^*$ whose length equals~$|\zeta^*|$.
This is a symmetric set with respect to the origin in the real line, whether or not multiplicities are counted. 
We enumerate the eigenvalues of~$\pd_{\T^2}$ in a non-decreasing order with multiplicities counted: 
\begin{align}\label{eq:spectrum-pd}
        \cdots \leq \lambda_{-k}\leq  \cdots \leq \lambda_{-1} <0=\lambda_{0,1}=\lambda_{0,2}<\lambda_1\leq \cdots \leq \lambda_k \cdots  
    \end{align}
with~$\lambda_{-k}=-\lambda_k$. 
Let~$\braces{\eta_k\mid k\in\mathbb{Z},k\neq 0} \cup \braces{\eta_{0,1},\eta_{0,2}}$ be a complete orthonormal basis of~$L^2(\T^2,\C^2)$ consisting of eigenspinors:
\begin{align}
    \pd_{\T^2}\eta_k = \lambda_k \eta_k, \qquad \forall k\neq 0
\end{align}
\begin{align}
    \pd_{\T^2}\eta_{0,j}=\lambda_{0,j}\eta_{0,j}=0, \quad j=1,2. 
\end{align}
Then, similar to the discussion in~\cite{WZ2026Stationary}, we have 
\begin{itemize}
    \item[(1)] $\Spect(\pd_{\T^2})\subset\Spect(\D_{\T^2})$, and for each~$\lambda_k\in\Spect(\pd_{\T^2})\setminus \braces{0}$,
    \begin{align}
        \begin{pmatrix} \eta_k \\ \eta_k \end{pmatrix}, 
        \begin{pmatrix} \eta_{-k} \\ -\eta_{-k} \end{pmatrix}
        \in \Eigen(\D_{\T^2};\lambda_k)
    \end{align}
    and the above two spinors are linearly independent (but not normalized).
    Meanwhile the~$0$-eigenspinors of~$\D_{\T^2}$ are given by the constant sections 
    \begin{align}
        \psi_{0,1}=\begin{pmatrix} 1 \\ 0 \\ 0 \\ 0 \end{pmatrix}, \quad 
        \psi_{0,2}=\begin{pmatrix} 0 \\ 1 \\ 0 \\ 0 \end{pmatrix}, \quad 
        \psi_{0,3}=\begin{pmatrix} 0 \\ 0 \\ 1 \\ 0 \end{pmatrix}, \quad 
        \psi_{0,4}=\begin{pmatrix} 0 \\ 0 \\ 0 \\ 1 \end{pmatrix}. 
    \end{align}

    \item[(2)] $\Spect(\D_{\T^2})\subset\Spect(\pd_{\T^2})$, and if~$\D_{\T^2}\psi=\lambda\psi$, then
    \begin{align}
        \pd_{\T^2}(\psi_{\up}+\psi_{\down}) = \lambda (\psi_{\up}+\psi_{\down}), & & 
        \pd_{\T^2}(\psi_{\up}-\psi_{\down}) =-\lambda (\psi_{\up}-\psi_{\down}).
    \end{align}
    Together with the construction in (1) we see that both~$\lambda$ and~$-\lambda$ are in~$\Spect(\pd_{\T^2})$. 
\end{itemize}
Moreover, the operator~$\D_{\T^2}$ is unitarily equivalent to~$\pd_{\T^2}\oplus (-\pd_{\T^2})$ under the transform 
\begin{align}
    \begin{pmatrix}\psi_{\up} \\ \psi_{\down} \end{pmatrix}
    \mapsto 
    \frac{1}{\sqrt{2}}
    \begin{pmatrix} \psi_{\up}+\psi_{\down} \\ \psi_{\up}-\psi_{\down} \end{pmatrix}. 
\end{align}
Recalling~$\Spect(\pd_{T^2})$ is symmetric in~$\R$, we thus have 
\begin{align}
    \Eigen(\D_{\T^2};\lambda) \cong \Eigen(\pd_{\T^2};\lambda)\oplus\Eigen(\pd_{\T^2};-\lambda), & & 
    \Mtp(\D_{\T^2};\lambda)= 2\Mtp(\pd_{\T^2};\lambda).
\end{align}
This proves the first half of the following lemma. 

\begin{lemma}\label{lemma:spectrum}
    Let~$\mathbb{T}^2=\faktor{\R^2}{\Gamma(\ell_1,\ell_2)}$ be equipped with the trivial spin structure, and let~$\D_{\T^2}, \pD_{\T^2}$ be defined as in~\eqref{eq:Dirac-operators}. 
    \begin{itemize}
        \item[(i)] The spectrum of~$\D_{\T^2}$ is 
        \begin{align}
            \Spect(\D_{\T^2})= \Spect(\pd_{\T^2})
            =\braces{\pm 2\pi |\zeta^*| \mid \zeta^*\in \Gamma(\ell_1,\ell_2)^*\subset\R^2}
        \end{align}
        and for each~$\lambda\in \Spect(\pd_{\T^2})$,~$\Mtp(\D_{\T^2};\lambda)= 2\Mtp(\pd_{\T^2};\lambda)$.

        \item[(ii)] The spectrum of~$\pD_{\T^2}$ is 
            \begin{align}
                \Spect(\pD_{\T^2})=\braces{ \pm \sqrt{\zeta^2+\lambda^2+m^2} \mid \lambda \in \Spect(\D_{\T^2})} 
            \end{align}
            with multiplicities given by
            \begin{align}
                \Mtp(\pD_{\T^2};\sqrt{\zeta^2+\lambda^2+m^2})
                =\Mtp(\pD_{\T^2};-\sqrt{\zeta^2+\lambda^2+m^2})
                =\Mtp(\D_{\T^2}; \lambda),
            \end{align}
            for any~$\lambda\in \Spect(\D_{\T^2})\setminus \braces{0}$, while 
            \begin{align}
                \Mtp(\pD_{\T^2};\sqrt{\zeta^2+m^2})
                =\Mtp(\pD_{\T^2};-\sqrt{\zeta^2+m^2})=2.
            \end{align}
    \end{itemize}
    
\end{lemma}

\begin{proof}
    It remains to determine the spectrum of~$\pD_{\T^2}$.

    First note that the operator~$\D_{\T^2}^2$ is an elliptic operator of Laplacian type. 
    It is also self-adjoint on a closed manifold. 
    Every eigenspinor of~$\D_{\T^2}$ is automatically an eigenspinor of~$\D_{\T^2}^2$.
    Since such eigenspinors form a complete basis for the~$L^2$ space of sections of the bundle~$\Sigma\T^2\oplus \Sigma\T^2$, we see that
    \begin{align}
        \Spect(\D_{\T^2}^2)=\braces{\lambda^2 \mid \lambda\in\Spect(\D_{\T^2})}
    \end{align}
    and for~$\lambda\neq 0$, 
    \begin{align}\label{eq:Eigen-lambda2}
        \Eigen(\D_{\T^2}^2;\lambda^2)
        =\Eigen(\D_{\T^2};\lambda)
        \oplus
        \Eigen(\D_{\T^2};-\lambda),
    \end{align}
    \begin{align}
        \Mtp(\D_{\T^2}^2;\lambda^2)
        = \Mtp(\D_{\T^2};\lambda) 
        + \Mtp(\D_{\T^2};-\lambda)
        =2\Mtp(\D_{\T^2};\lambda), 
    \end{align}
    while~$\Mtp(\D_{\T^2}^2;0)=4=\Mtp(\D_{\T^2};0)$.

    As a self-adjoint elliptic operator on a closed manifold, the operator~$\pD_{\T^2}$ also has only eigenvalues of finite multiplicities. 
    Suppose~$\rho\in\Spect(\pD_{\T^2})$ with~$\pD_{\T^2}\psi=\rho\psi$.
    Then 
    \begin{align}
        \rho^2\psi=\pD_{\T^2}^2\psi 
        = \D_{\T^2}^2\psi + (\zeta^2+m^2)\psi. 
    \end{align}
    Thus~$\D_{\T^2}^2\psi= (\rho^2-\zeta^2-m^2)\psi=\lambda^2\psi$ for some~$\lambda\in\Spect(\D_{\T^2})$. 
    It follows that
    \begin{align}
        \rho\in \braces{\pm \sqrt{\zeta^2+\lambda^2+m^2}}. 
    \end{align}
    Next we show that both of~$\pm \sqrt{\zeta^2+\lambda^2+m^2} $ appear in the spectrum of~$\pD_{\T^2}$, with the same multiplicity as~$\lambda\in\Spect(\D_{\T^2})$.

    For each~$\lambda>0$, consider the linear map
    \begin{align}
        \frac{\zeta\gamma^0\gamma^3+m\gamma^0}{\sqrt{\zeta^2+m^2}}\colon \Eigen(\D_{\T^2};\lambda)\to \Eigen(\D_{\T^2};-\lambda).
    \end{align}
    It is well-defined: if~$\D_{\T^2}\psi=\lambda\psi$, then 
    \begin{align}
        \D_{\T^2}\parenthesis{\frac{1}{\sqrt{\zeta^2+m^2}}(\zeta\gamma^0\gamma^3+m\gamma^0)\psi}
        =&-\frac{1}{\sqrt{\zeta^2+m^2}}\parenthesis{\zeta\gamma^0\gamma^3+m\gamma^0}
        \D_{\T^2}\psi \\
        =&-\lambda \parenthesis{\frac{1}{\sqrt{\zeta^2+m^2}}(\zeta\gamma^0\gamma^3+m\gamma^0)\psi}. 
    \end{align}
    It is also invertible since 
    \begin{align}
        \parenthesis{ \frac{\zeta\gamma^0\gamma^3+m\gamma^0}{\sqrt{\zeta^2+m^2}} }^2=I_4.
    \end{align}
    Let~$d\equiv \Mtp(\D_{\T^2};\lambda)$ and choose an orthonormal basis~$\psi_1,\dots, \psi_d$ of~$\Eigen(\D_{\T^2};\lambda)$. 
    Set
    \begin{align}
        \phi_j \coloneqq \frac{1}{\sqrt{\zeta^2+m^2}} \parenthesis{\zeta\gamma^0\gamma^3+m\gamma^0}\psi_j,\quad 1\leq j\leq d. 
    \end{align}
    Then~$\phi_1,\dots,\phi_d$ form a basis of~$\Eigen(\D_{\T^2};-\lambda)$. 
    Moreover,
    \begin{align}
        \parenthesis{\zeta\gamma^0\gamma^3+m\gamma^0}\psi_j= \sqrt{\zeta^2+m^2}\phi_j, 
        & & 
        \parenthesis{\zeta\gamma^0\gamma^3+m\gamma^0}\phi_j = \sqrt{\zeta^2+m^2}\psi_j.
    \end{align}
    Consequently,
    \begin{align}
        \pD_{\T^2}\psi_j = \lambda \psi_j+\sqrt{\zeta^2+m^2}\phi_j,
        & & 
        \pD_{\T^2}\phi_j = \sqrt{\zeta^2+m^2}\psi_j-\lambda \phi_j.
    \end{align}
    Hence the two-dimensional space~$\Span_{\C}\{\psi_j,\phi_j\}$ is invariant under~$\pD_{\T^2}$,
    on which~$\pD_{\T^2}$ acts via a matrix
    \begin{align}
        \begin{pmatrix}
            \lambda & \sqrt{\zeta^2+m^2}\\
            \sqrt{\zeta^2+m^2} & -\lambda
        \end{pmatrix}.
    \end{align}
    Its characteristic polynomial is~$p(\mu)=\mu^2-\lambda^2-\zeta^2-m^2$, and thus its eigenvalues are
    \begin{align}
        \mu=\pm\sqrt{\lambda^2+\zeta^2+m^2}.
    \end{align}

    Recall~$\Eigen(\D_{\T^2}^2;\lambda^2)$ is decomposed as in~\eqref{eq:Eigen-lambda2}.
    The above construction decomposes this space into~$d$ two-dimensional invariant subspaces. Hence each of the eigenvalues~$\pm\sqrt{\lambda^2+\zeta^2+m^2}$ has multiplicity~$d$.  
    Hence
    \begin{align}
        \Mtp\parenthesis{
            \pD_{\T^2};
            \sqrt{\lambda^2+\zeta^2+m^2}}
        =\Mtp\parenthesis{
            \pD_{\T^2};
            -\sqrt{\lambda^2+\zeta^2+m^2}}
        =\Mtp(\D_{\T^2};\lambda).
    \end{align}

    As for the eigenspaces for~$\pm\sqrt{\zeta^2+m^2}$, we simply note that they must arise from~$\Eigen(\D_{\T^2};0)$ and~$\pD_{\T^2}$ acting on constant spinors reduces to the matrix
    \begin{align}
        \mathscr{D}\equiv
        \begin{pmatrix}
            m & 0 & \zeta & 0 \\
            0 & m & 0& -\zeta \\
            \zeta & 0 & -m & 0 \\
            0 & -\zeta & 0 & -m
        \end{pmatrix}. 
    \end{align}
    This matrix has eigenvalues~$\pm\sqrt{\zeta^2+m^2}$ with eigenspaces 
    \begin{align}
        \Eigen(\mathscr{D};\sqrt{\zeta^2+m^2})
        =\Span\braces{ \begin{pmatrix} 1 \\ 0 \\ \frac{\zeta}{ m+\sqrt{\zeta^2+m^2} } \\ 0  \end{pmatrix}, 
        \begin{pmatrix} 0 \\ 1 \\ 0 \\ \frac{-\zeta}{m+\sqrt{\zeta^2+m^2}} \end{pmatrix} }, 
    \end{align}
    \begin{align}
        \Eigen(\mathscr{D};-\sqrt{\zeta^2+m^2})
        =\Span\braces{ \begin{pmatrix} \frac{-\zeta}{m+\sqrt{\zeta^2+m^2}} \\ 0 \\ 1 \\ 0  \end{pmatrix}, 
        \begin{pmatrix} 0 \\ \frac{\zeta}{m+\sqrt{\zeta^2+m^2}} \\ 0 \\ 1 \end{pmatrix}}. 
    \end{align}
    Note that these constant eigenspinors reduce to the~$\psi_{0,j}$'s when~$\zeta=0$. 
    This proves the second part of the lemma. 
\end{proof}

\section{Null-cone separation and a resolvent estimate}\label{sect:key lemma}

The space~$\Eigen(\pD_{\T^2}, \sqrt{\zeta^2+m^2})$ consists of constant spinors on $\mathbb{T}^2$, but such spinors no longer lie in the $(+1)$-eigenspace of $\gamma^0$, a property used in the contradiction argument in~\cite{WZ2026Stationary}.
Nevertheless, they cannot lie in the Lorentz null cone.  
Indeed, any~$e\in\Eigen(\pD_{\T^2};\sqrt{\zeta^2+m^2})$ satisfies
\begin{align}
    \bar{e}e=\frac{m}{\sqrt{m^2+\zeta^2}}|e|^2.
\end{align}
which is nonzero unless~$e=0$. 
In particular,~$\mathcal{N}\cap \Eigen(\pD_{\T^2};\sqrt{\zeta^2+m^2})=\braces{0}$.  

Let~$\proj_{m,\zeta}$ be the~$L^2$-orthogonal projection to~$\Eigen(\pD_{\T^2};\sqrt{m^2+\zeta^2})$ and~$\proj_{m,\zeta}^\perp = \id -\proj_{m,\zeta}$ the complementary projection.
The following lemma gives a quantitative separation of the null cone from the lowest positive eigenspace~$\Eigen(\pD_{\T^2};\sqrt{\zeta^2+m^2})$. 
\begin{lemma}\label{lemma:null-cone-separation}
    For any~$\nu\in (1,2)$, there exist~$\kappa_1>0$ and~$\kappa_2>0$ such that
    \begin{align}\label{eq:kappa}
        \|\varphi\|_{L^4}=1,\quad \|\bar{\varphi} \varphi\|_{L^\nu}\leq\kappa_1 \quad\Longrightarrow\quad \|\proj_{m,\zeta}^\perp \varphi\|_{L^4}\geq\kappa_2.
    \end{align}
\end{lemma}
\begin{proof}
    The corresponding statement for~$\nu\in(1,\frac{3}{2})$ was proved in~\cite{WZ2026Stationary}; but the argument works also for~$\nu\in (1,2)$.  
    Since these constants are crucial for us, we give the detailed proof here. 
    
    To see~\eqref{eq:kappa}, we argue by contradiction. 
    Suppose  there exists a sequence~$(\varphi_k)$ such that
    \begin{align}
        \|\varphi_k\|_{L^4}=1, \quad \|\bar{\varphi}_k\varphi_k\|_{L^\nu} \to 0,
        \quad \|\proj_{m,\zeta}^\perp \varphi_k\|_{L^4}\to 0, \qquad \mbox{ as } \; k\to+\infty.
    \end{align}
    We observe that ~$\|\proj_{m,\zeta}\varphi_k\|_{L^4}\leq C\|\varphi_k\|_{L^4}=C$, with~$C$ depending on~$\T^2$ and the finite-dimensional space~$\Eigen(\pD_{\T^2};\sqrt{\zeta^2+m^2})$. 
    Since any two norms on a finite-dimensional space are equivalent, the sequence~$(\proj_{m,\zeta}\varphi_k)$ is bounded in~$\Eigen(\pD_{\T^2};\sqrt{m^2+\zeta^2})\cong\C^2$. 
    Therefore, upon passing to a subsequence if necessary, the sequence~$\proj_{m,\zeta}\varphi_k$ converges to some~$\varphi_{*}$.
    Together with the convergence~$\|\proj_{m,\zeta}^\perp \varphi_k\|_{L^4}\to0$, we see that~$\varphi_k$ converges in~$L^4$ to~$\varphi_*$ with~$\|\varphi_{*}\|_{L^4}=1$, and~$\bar{\varphi}_k\varphi_k$ converges in~$L^2$ to~$\bar{\varphi}_*\varphi_*$. 
    Since~$\nu<2$, we see that~$\bar{\varphi}_k\varphi_k$ also converges in~$L^\nu$ to~$\bar{\varphi}_*\varphi_*$, hence~$\bar{\varphi}_*\varphi_*=0$ a.e. 
    This is impossible, since~$\varphi_*\in\Eigen(\pD_{\T^2};\sqrt{m^2+\zeta^2})$ and~$\bar{\varphi}_*\varphi_*= \frac{m}{\sqrt{\zeta^2 + m^2}}|\varphi_*|^2$.
    
\end{proof}
      We remark that the above statement actually holds for any~$\nu\geq 1$ and there is no need for the requirement~$\nu<2$.
      This follows by a H\"older inequality argument. 
      Indeed, there holds an even more general principle. 
      We state it here for future reference. 

\begin{lemma}[An abstract separation lemma]\label{lemma:abstract-separation}
    Let~$X$ be a normed linear space and let~$E\subset X$ be a finite-dimensional subspace.
    Let~$(Z,d_Z)$ be a metric space with a distinguished point~$0_Z\in Z$, and let
    \begin{align}
        \Phi\colon X\to Z
    \end{align}
    be continuous.
    Assume that
    \begin{align}\label{eq:trivial-intersection}
        E\cap \Phi^{-1}(0_Z)=\braces{0}.
    \end{align}
    Suppose moreover that~$\rho\colon X\to[0,+\infty)$ has the property that, for every sequence~$(u_n)\subset X$,
    \begin{align}\label{eq:rho-Phi}
        \rho(u_n)\to0
        \qquad\Longrightarrow\qquad
        d_Z(\Phi(u_n),0_Z)\to0.
    \end{align}
    Then there exist constants~$\kappa_1,\kappa_2>0$ such that
    \begin{align}\label{eq:abstract-separation}
        \|u\|_X=1,
        \qquad
        \rho(u)\leq\kappa_1
        \qquad\Longrightarrow\qquad
        \dist_X(u,E)\geq\kappa_2.
    \end{align}
\end{lemma}

\begin{proof}
    Suppose not, then there would exist a sequence~$(u_n)\subset X$ such that
    \begin{align}
        \|u_n\|_X=1,
        \qquad
        \rho(u_n)\leq\frac{1}{n},
        \qquad
        \dist_X(u_n,E)\leq\frac{1}{n}.
    \end{align}
    Choose~$e_n\in E$ such that~$\|u_n-e_n\|_X\leq \frac{2}{n}$. 
    In particular,~$\|e_n\|_X\to1$. 
    Since~$E$ is finite dimensional, after passing to a subsequence if necessary, there exists~$e_*\in E$ such that~$e_n\to e_*$ in~$X$. 
    Consequently,
    \begin{align}
        u_n\to e_*
        \qquad\mbox{ in }X,
        \qquad
        \|e_*\|_X=1.
    \end{align}
    By the continuity of~$\Phi$,
    \begin{align}
        d_Z(\Phi(u_n),\Phi(e_*))\to0.
    \end{align}
    On the other hand,~$\rho(u_n)\to0$, and hence~\eqref{eq:rho-Phi} gives~$ d_Z(\Phi(u_n),0_Z)\to0$. 
    By uniqueness of limits in the metric space~$Z$, we must have~$\Phi(e_*)=0_Z$. 
    Thus~$e_*\in E\cap\Phi^{-1}(0_Z)=\braces{0}$, contradicting~$\|e_*\|_X=1$.
\end{proof}

\begin{rmk}
    To apply the above general principle to our case, we let
    \begin{align}
        X=L^p(\T^2,\C^4),
        \qquad
        \Phi(\varphi)=\bar{\varphi}\varphi,
        \qquad
        \rho(\varphi)=\|\bar{\varphi}\varphi\|_{L^q}.
    \end{align}
    with~$1\leq p,q<+\infty$. 
    Set
    \begin{align}
        r\coloneqq \min\braces{q,\frac{p}{2}}.
    \end{align}
    If~$r\geq1$, we take~$Z=L^r(\T^2)$ with its usual norm metric, while if~$0<r<1$, we equip~$L^r(\T^2)$ with the metric
    \begin{align}
        d_r(f,g)\coloneqq \int_{\T^2}|f-g|^r\dv_{\T^2}.
    \end{align}
    Since~$r\leq q$ and~$\T^2$ has finite volume,
    \begin{align}
        \|\bar{\varphi}_n\varphi_n\|_{L^q}\to0
        \qquad\Longrightarrow\qquad
        \bar{\varphi}_n\varphi_n\to0
        \quad\mbox{ in }Z.
    \end{align}
    Moreover, since~$2r\leq p$, the map
    \begin{align}
        \Phi\colon L^p(\T^2,\C^4)\to Z,
        \qquad
        \varphi\mapsto\bar{\varphi}\varphi,
    \end{align}
    is continuous.
    Hence, for every finite-dimensional subspace~$E\subset L^p(\T^2,\C^4)$ satisfying
    \begin{align}
        E\cap\braces{\varphi\mid \bar{\varphi}\varphi=0\;\mbox{ a.e.}}=\braces{0},
    \end{align}
    there exist~$\kappa_1,\kappa_2>0$ such that
    \begin{align}
        \|\varphi\|_{L^p}=1,
        \qquad
        \|\bar{\varphi}\varphi\|_{L^q}\leq\kappa_1
        \qquad\Longrightarrow\qquad
        \dist_{L^p}(\varphi,E)\geq\kappa_2.
    \end{align}
    In our case,~$E=\Eigen(\pD_{\T^2};\sqrt{\zeta^2+m^2})$, and~$\dist_{L^p}(\varphi,E)\leq \|\proj_{m,\zeta}^\perp \varphi\|_{L^p}$, so the general lemma applies. 
\end{rmk}
Note that the above compactness argument is qualitative and does not provide explicit values of~$\kappa_1$ and~$\kappa_2$.

\ 

      We also need a resolvent estimate for~$\pD_{\T^2}-a $, for any~$a\in [0,\sqrt{\zeta^2+m^2}]$, on the orthogonal complement of the lowest positive eigenspace~$\Eigen(\pD_{\T^2};\sqrt{\zeta^2+m^2})$.

\begin{lemma}\label{lemma:invertibility-orthogonal}
    There exists a constant~$C=C(m,\zeta,\mathbb{T}^2)$ such that for any~$a\in [0,\sqrt{\zeta^2+m^2}]$, 
    \begin{align}
        \|\proj_{m,\zeta}^\perp \varphi\|_{L^4}\leq C\|\proj_{m,\zeta}^\perp(\pD_{\T^2}-a)\varphi\|_{L^{\frac{4}{3}}}.
    \end{align}
\end{lemma}

\begin{proof}
    Suppose not, then there exist sequences~$(a_k)\subset [0,\sqrt{\zeta^2+m^2}]$ and~$(\varphi_k)\subset W^{1,\frac{4}{3}}$ such that
        \begin{align}
            \|\proj_{m,\zeta}^\perp\varphi_k\|_{L^4}=1, & &
            \|\proj_{m,\zeta}^\perp(\pD_{\T^2}-a_k)\varphi_k\|_{L^{\frac{4}{3}}}\to 0.
        \end{align}
        We may assume~$a_k\to a_*\in [0,\sqrt{\zeta^2+m^2}]$ by Heine--Borel theorem. 
        Note that for any~$k$,
        \begin{align}
            \|\proj_{m,\zeta}^\perp \varphi_k \|_{W^{1,\frac{4}{3}}}
            \leq& C \parenthesis{ \| \proj_{m,\zeta}^\perp \varphi_k\|_{L^{\frac{4}{3}}} + \|\pD_{\T^2}\proj_{m,\zeta}^\perp \varphi_k\|_{L^{\frac{4}{3}}} } \\
            \leq& C \parenthesis{
            \| \proj_{m,\zeta}^\perp \varphi_k\|_{L^{\frac{4}{3}}}
            + \|(\pD_{\T^2}-a_k)\proj_{m,\zeta}^\perp \varphi_k\|_{L^{\frac{4}{3}}} 
            + a_k\|\proj_{m,\zeta}^\perp \varphi_k\|_{L^{\frac{4}{3}}} } \\
            \leq& C \parenthesis{ (\sqrt{m^2+\zeta^2}+1)\| \proj_{m,\zeta}^\perp \varphi_k\|_{L^{\frac{4}{3}}} + \|\proj_{m,\zeta}^\perp (\pD_{\T^2}-a_k)\varphi_k\|_{L^{\frac{4}{3}}} },
        \end{align}
        where~$C$ is independent of~$k$ and we have used the commutativity of~$\proj_{m,\zeta}^\perp$ with~$(\pD_{\T^2}-a_k)$ in the last inequality.
        Hence the sequence~$(\proj_{m,\zeta}^\perp \varphi_k)$ is also uniformly bounded in~$W^{1,\frac{4}{3}}$ and we may assume
        \begin{align}
            \proj_{m,\zeta}^\perp\varphi_k \to \varphi_*=\proj_{m,\zeta}^\perp \varphi_*\quad  \mbox{ in } \;  L^{\frac{4}{3}}.
        \end{align}
        Moreover, from
        \begin{align}
            \pD_{\T^2}(\proj_{m,\zeta}^\perp\varphi_k)= \proj_{m,\zeta}^\perp (\pD_{\T^2}-a_k)\varphi_k+ a_k \proj_{m,\zeta}^\perp\varphi_k \xrightarrow{L^{\frac{4}{3}}} 0+ a_*\varphi_*
        \end{align}
        it follows that
        \begin{align}
            \pD_{\T^2}\varphi_*=a_*\varphi_*.
        \end{align}
        Since~$\varphi_* =\proj_{m,\zeta}^\perp\varphi_*\perp \Eigen(\pD_{\T^2};\sqrt{m^2+\zeta^2})$ while~$\dist(a_*, \Spect(\pD_{\T^2})\setminus\braces{\sqrt{m^2+\zeta^2}})>0$, we must have~$\varphi_*=0$.
        This in turn implies that
        \begin{align}
            \|\proj_{m,\zeta}^\perp \varphi_k\|_{W^{1,\frac{4}{3}}} \to 0
        \end{align}
        as~$k\to+\infty$, contradicting~$\|\proj_{m,\zeta}^\perp\varphi_k\|_{L^4}=1$ for any~$k$.
\end{proof}

Another explanation of this resolvent estimate will be given in Remark~\ref{rmk:resolvent-estimate}, using Fourier modes, and we can make the constant~$C$ more explicit.

\section{The variational framework}\label{sect:variational framework}

In this section we build up the working space~$H^{\frac{1}{2}}(\T^2;\C^4)$ and realize the equation~\eqref{eq:NDE-T2} as the Euler-Lagrange equation of a~$C^1$ functional, so that we can obtain nontrivial solutions via minimax method in the sequel. 

As explained above, we consider the bundle~$\Sigma\T^2\oplus\Sigma\T^2\to\T^2$, and work with sections of regularity~$H^{1/2}$.  
This can be defined using the operator~$\pD_{\T^2}$, the resulting spaces agree with the usual fractional Sobolev spaces, see e.g.~\cite{Ammann2003Habil, SireXu2023variational} for more details. 
The construction is analogous to that in~\cite{WZ2026Stationary}; we include the details needed below.

To the self-adjoint elliptic operator~$\pD_{\T^2}$ we associate the spectral measure
\begin{align}
    \tau\colon \mathcal{B}(\R)\to \proj(L^2(\T^2;\C^4))
\end{align}
Then we have 
\begin{align}
    \pD_{\T^2} = \int_{\R} \lambda \dd\tau(\lambda)
\end{align}
and the absolute value operator of~$\pD_{\T^2}$ is given formally by 
\begin{align}
    |\pD_{\T^2}|=\int_\R |\lambda| \dd\tau(\lambda).
\end{align}
For any~$s\geq 0$, we can define the fractional powers~$|\pD_{\T^2}|^s$ by 
\begin{align}
    |\pD_{\T^2}|^s = \int_{\R} |\lambda|^s \dd\tau(\lambda)
\end{align}
whose domain is the space
\begin{align}
    \Dom(|\pD_{\T^2}|^s)
    =&\braces{ \varphi\in L^2(\T^2,\C^4) \mid \int_\R |\lambda|^{2s} \dd\Abracket{\tau(\lambda)\varphi,\varphi } <+\infty}  \\
    =&H^{s}(\mathbb{T}^2;\C^4) 
    =H^s(\T^2;\Sigma\T^2\oplus\Sigma\T^2).
\end{align}
The space~$H^{-s}(\T^2,\C^4)$ denotes the dual space of~$H^{s}(\T^2,\C^4)$ as usual. 

\ 

In our situation we need to consider~$s=\frac{1}{2}$, which is the most natural space to deal with Dirac equations. 
The working space is 
\begin{align}
    H^{\frac{1}{2}}(\T^2,\C^4)
    =\braces{\varphi\in L^2(\T^2,\C^4) \; \mid \; \|\varphi\|_{H^{\frac{1}{2}}} \equiv \parenthesis{ \|\varphi\|_{L^2}^2 + \| |\pD_{\T^2}|^{\frac{1}{2}} \varphi \|_{L^2}^2 }^{\frac{1}{2}} <+\infty  }. 
\end{align}
Since~$m>0$, the following norm 
\begin{align}
    \tnorm{\varphi}_{H^{\frac{1}{2}}(\T^2,\C^4)}\coloneqq \| |\pD_{\T^2}|^{\frac{1}{2}}\varphi\|_{L^2(\T^2,\C^4)}
\end{align}
is equivalent to the standard norm~$\|\cdot\|_{H^{\frac{1}{2}}}$. 

As usual we decompose the Hilbert space~$H^{\frac{1}{2}}$ as the direct sum of the subspaces spanned by eigenspinors of positive respectively negative eigenvalues
\begin{align}
    H^{\frac{1}{2}}(\mathbb{T}^2,\C^4)= H^{\frac{1}{2},+}\oplus H^{\frac{1}{2},-}, 
\end{align}
with 
\begin{align}
    H^{\frac{1}{2},+}\coloneqq \overline{ \bigoplus_{k>0} \Eigen(\pD_{\T^2};\mu_k) }^{H^{\frac{1}{2}}}, \quad 
    H^{\frac{1}{2},-}\coloneqq \overline{ \bigoplus_{k<0} \Eigen(\pD_{\T^2};\mu_k) }^{H^{\frac{1}{2}}} . 
\end{align}
The orthogonal projections onto these subspaces are denoted by~$P^{\pm}$. 
With these notations, we have 
\begin{align}
    \varphi=P^+\varphi + P^-\varphi \equiv \varphi^+ + \varphi^- \in H^{\frac{1}{2},+}\oplus H^{\frac{1}{2},-}=H^{\frac{1}{2}}, 
\end{align}
\begin{align}
    \pD_{\T^2}=&\pD_{\T^2}(P^+ + P^-) = \pD_{\T^2} P^+ + \pD_{\T^2} P^-,  \\
    |\pD_{\T^2}|=& \pD_{\T^2} P^+ - \pD_{\T^2} P^-. 
\end{align}
Note that~$\pD_{\T^2}$ is a first-order elliptic operator, and for a spinor~$\varphi\in H^{\frac{1}{2}}(\T^2,\C^4)$, 
\begin{align}
    \pD_{\T^2}\varphi = \pD_{\T^2}P^+\varphi + \pD_{\T^2}P^-\varphi 
    =|\pD_{\T^2}|P^+\varphi - |\pD_{\T^2}|P^-\varphi 
\end{align}
is only an element in~$H^{-\frac{1}{2}}$.
Thus~$\Abracket{\pD_{\T^2}\varphi,\varphi}$ cannot be understood as a~$L^2$ function. 
However, as usually done in analysis, we can make sense of it via the duality pairing~$\Abracket{\pD_{\T^2}\varphi,\varphi}_{H^{-1/2}\times H^{1/2}}\in \R$, and write 
\begin{align}
    \int_{\T^2}\Abracket{\pD_{\T^2}\varphi,\varphi}\dv_{\T^2}
    \coloneqq 
    \Abracket{\pD_{\T^2}\varphi,\varphi}_{H^{-1/2}\times H^{1/2}}\in \R. 
\end{align}
In terms of the above decompositions, we have 
\begin{align}
    \int_{\T^2}\Abracket{\pD_{\T^2}\varphi,\varphi}\dv_{\T^2}
    \coloneqq & 
    \Abracket{|\pD_{\T^2}|P^+\varphi - |\pD_{\T^2}|P^-\varphi ,\varphi}_{H^{-1/2}\times H^{1/2}} \\
    =&\| |\pD_{\T^2}|^{\frac{1}{2}}P^+\varphi\|_{L^2}^2 
     - \| |\pD_{\T^2}|^{\frac{1}{2}}P^-\varphi\|_{L^2}^2
\end{align}
where the last expression is well-defined variationally.  
In the same way, we have 
\begin{align}
    \tnorm{\varphi}_{H^{\frac{1}{2}}}^2 = \| |\pD_{\T^2}|^{\frac{1}{2}}\varphi \|_{L^2}^2 
    =&\| |\pD_{\T^2}|^{\frac{1}{2}}P^+\varphi\|_{L^2}^2 
     + \| |\pD_{\T^2}|^{\frac{1}{2}}P^-\varphi\|_{L^2}^2 \\
     =& \Abracket{|\pD_{\T^2}|P^+\varphi +|\pD_{\T^2}|P^-\varphi ,\varphi}_{H^{-1/2}\times H^{1/2}} \\
     \equiv & \int_{\T^2}\Abracket{|\pD_{\T^2}|\varphi,\varphi}\dv_{\T^2}
\end{align}
where the last integral is again understood in the sense of the duality pairing.

\ 

Equivalently, since~$\Spect(\pD_{\T^2})$ has only point spectrum consisting of eigenvalues of finite multiplicity, one could use the Fourier expansions via eigenspinors and define~$|\pD_{\T^2}|^s$ and its domain as follows. 
List the eigenvalues of~$\pD_{\T^2}$ in an increasing order as 
\begin{align}
    -\infty \leftarrow \cdots \leq 
    \mu_{-k}\leq \cdots \leq \mu_{-1} <0 < \mu_1 \leq \cdots \leq \mu_k \leq \cdots \to +\infty
\end{align}
and let~$\braces{ \chi_k\mid k \in\mathbb{Z}_*}$ (where~$\mathbb{Z}_*$ denotes the set of nonzero integers) be a~$L^2$-orthonormal basis consisting of eigenspinors. 
Then any~$L^2$-section~$\varphi$ of the bundle~$\Sigma\T^2 \oplus \Sigma\T^2$ can be expanded in Fourier modes as 
\begin{align}
    \varphi=\sum_{k\in \mathbb{Z}_*} a_k \chi_k. 
\end{align}
Then~$\pD_{\T^2}$ acts on~$\varphi$ as
\begin{align}
    \pD_{\T^2}\varphi=\sum_{k\in\mathbb{Z}_*} \mu_k a_k \chi_k.
\end{align}
For any~$s\in\R$, we define~$H^s(\T^2,\C^4)$ by the Fourier expansion
\begin{align}
    H^s(\T^2,\C^4)
    =
    \braces{
        \varphi=\sum_{k\in\mathbb Z_*}a_k\chi_k
        \;\middle|\;
        \sum_{k\in\mathbb Z_*}|\mu_k|^{2s}|a_k|^2<+\infty
    },
\end{align}
where for~$s<0$ the space is understood as the completion of finite
Fourier sums with respect to the corresponding norm.
We define
\begin{align}
    |\pD_{\T^2}|^s\varphi
    =
    \sum_{k\in\mathbb Z_*}|\mu_k|^s a_k\chi_k,
\end{align}
with domain exactly~$H^{s}(\T^2,\C^4)$. Then
\begin{align}
    \tnorm{\varphi}_{H^s}
    =
    \||\pD_{\T^2}|^s\varphi\|_{L^2}
    =
    \parenthesis{
        \sum_{k\in\mathbb Z_*}|\mu_k|^{2s}|a_k|^2
    }^{\frac12},
\end{align}
for every~$s\in\R$.
In this way we can get a clearer picture of the fractional Sobolev spaces on~$\T^2$.
For~$s=\frac{1}{2}$, the subspace~$H^{\frac{1}{2},+}$ then consists of 
\begin{align}
    \varphi=\sum_{k>0} a_k \chi_k, & & \mbox{ with } & & \sum_{k>0}\mu_k |a_k|^2<+\infty,
\end{align}
and an analogous description for~$H^{\frac{1}{2},-}$ holds.

\begin{rmk}\label{rmk:resolvent-estimate}
    We can now give another explanation of Lemma~\ref{lemma:invertibility-orthogonal}. 
    Using the Sobolev embeddings
    \begin{align}
    H^{\frac{1}{2}}(\T^2)\hookrightarrow L^4(\T^2,\C^4), & & L^{\frac{4}{3}}(\T^2,\C^4)\hookrightarrow H^{-\frac{1}{2}}(\T^2,\C^4),
    \end{align}
    it suffices to find a constant~$C>0$ such that
    \begin{align}
    \|\proj_{m,\zeta}^\perp \varphi\|_{H^{\frac{1}{2}}}\leq C\|\proj_{m,\zeta}^\perp(\pD_{\T^2}-a)\varphi\|_{H^{-\frac{1}{2}}}.
    \end{align}
    Recall that~$\proj_{m,\zeta}^\perp$ commutes with~$(\pD_{\T^2}-a)$.
    We thus restrict to spinors in the orthogonal complement of~$\Eigen(\pD_{\T^2};\sqrt{\zeta^2+m^2})$ and let 
    \begin{align}
    \varphi=\sum_{k\colon \mu_k\neq\sqrt{\zeta^2+m^2}}a_k \chi_k.
    \end{align}
    Then we have 
    \begin{align}
        \tnorm{\proj_{m,\zeta}^\perp(\pD_{\T^2}-a)\varphi }_{H^{-\frac{1}{2}}}^2
        =&\tnorm{ \sum_{k\colon \mu_k\neq\sqrt{\zeta^2+m^2}} a_k(\mu_k-a) \chi_k }_{H^{-\frac{1}{2}}}^2 \\
        =&\sum_{k\colon \mu_k\neq\sqrt{\zeta^2+m^2}}|\mu_k|^{-1}|(\mu_k-a)a_k|^2 \\
        =&\sum_{k\colon\mu_k\neq\sqrt{\zeta^2+m^2}} \left|\frac{\mu_k-a}{\mu_k}\right|^2 |\mu_k| |a_k|^2 \\
        \geq& \parenthesis{ \inf_{k\colon \mu_k\neq\sqrt{\zeta^2+m^2}} \left|\frac{\mu_k-a}{\mu_k}\right|^2  } \sum_{k\colon\mu_k\neq \sqrt{\zeta^2+m^2}} |\mu_k| |a_k|^2\\
        \geq& c_0^2 \tnorm{\proj_{m,\zeta}^\perp \varphi }_{H^{\frac{1}{2}}}^2.
    \end{align}
    In the last inequality, 
    \begin{align}
        0<c_0=1-\frac{\sqrt{\zeta^2+m^2}}{\sqrt{\zeta^2+m^2+\lambda_1^2}} 
        \leq 1-\frac{a}{\sqrt{\zeta^2+m^2+\lambda_1^2}}
        \leq \inf_{k\colon \mu_k\neq\sqrt{\zeta^2+m^2}} \left|\frac{\mu_k-a}{\mu_k}\right| 
    \end{align}
    for all~$a\in [0,\sqrt{\zeta^2+m^2}]$, where~$\lambda_1$ is the smallest positive eigenvalue of~$\pd_{\T^2}$ in~\eqref{eq:spectrum-pd}.
    This makes clear that the constant in Lemma~\ref{lemma:invertibility-orthogonal} is essentially~$c_0^{-1}$, multiplied by embedding constants.
\end{rmk}

\

With~$\zeta\in\bZ\frac{2\pi}{\ell_3}$ fixed, we are ready to consider the functional~$J_{\zeta}\colon H^{\frac{1}{2}}(\mathbb{T}^2;\C^4)\to \R$, defined by 
\begin{align}\label{eq:functional-T2}
    J_\zeta (\varphi)
    =&\int_{\T^2} \frac{1}{2}\Abracket{\pD_{\T^2} \varphi,\varphi }
     -\frac{a}{2}|\varphi|^2 -F(\varphi)\dv_{\T^2}
    \\
    =&\frac{1}{2}\parenthesis{ \tnorm{P^+\varphi}_{H^{1/2}}^2 
     - \tnorm{P^-\varphi }_{H^{1/2}}^2 
     -a\|\varphi\|_{L^2}^2 }
     -\int_{\T^2} F(\varphi)\dv_{\T^2}.
\end{align}
The Euler--Lagrange equation of~$J_\zeta$ is precisely given by~\eqref{eq:NDE-T2}.
We therefore seek nonzero critical points of~$J_{\zeta}$. 
However, the functional~$J_\zeta$ is strongly indefinite, and under the structural assumptions on the Soler-type nonlinearity considered here, the compactness required by the standard Palais--Smale or Cerami framework is not directly available. 
We therefore follow the perturbation strategy used in~\cite{WZ2026Stationary}: we first add a higher-order perturbation term, obtain critical points for the perturbed functionals, derive estimates that are uniform with respect to the perturbation parameter, and finally pass to the zero-perturbation limit.
Moreover, since~$\dim (\mathbb{T}^2)=2$ and~$H^{\frac{1}{2}}(\T^2)\hookrightarrow L^4(\T^2)$, we expect to deal with nonlinearities of subquartic growth.

More precisely, we consider 
\begin{align}\label{eq:perturbed-functional}
    J_{\zeta,\eps}(\varphi)
    =J_\zeta(\varphi)- \eps\int_{\T^2} |\varphi|^{\alpha_2}\dv_{\T^2},
\end{align}
where~$\eps\in[0,1]$ is the perturbation parameter. 
The Euler--Lagrange equation for~$J_{\zeta,\eps}$ is 
\begin{align}\label{eq:perturbed-NDE}
    \pD_{\T^2}\varphi-a\varphi - \p F(\varphi)= \eps \alpha_2 |\varphi|^{\alpha_2 -2 }\varphi, \quad \mbox{ on } \T^2. 
\end{align}
It reduces to~\eqref{eq:NDE-T2} when~$\eps= 0$. 
For later convenience we abbreviate~$F(\varphi)+\eps|\varphi|^{\alpha_2}$ as~$F_\eps(\varphi)$, so that~\eqref{eq:perturbed-NDE} can be rewritten as 
\begin{align}
    \pD_{\T^2}\varphi = a\varphi + \p F_\eps(\varphi).
\end{align}
By enlarging the constants~$A_j$ if necessary, we may assume that the perturbed nonlinearities~$F_\eps$ satisfy (F1-4) uniformly, for any~$\eps\in [0,1]$. 
Since~$\alpha_2\geq\alpha_1>2$, they also satisfy 
\begin{align}
    \p_\varphi F_\eps(x,\varphi)[\varphi] \geq \alpha_1 F_\eps(x,\varphi).
\end{align}

We make a notational remark here: formally we should write the functional as~$J_{\zeta,\eps,a}$ since at the end of Section~\ref{sect:perturbed solutions} we will vary~$a$ in a suitable range. 
But before that, the temporal frequency is fixed.
Thus we decide to use the notation~$J_{\zeta,\eps}$ and~$J_{\zeta,\eps, a}$ interchangeably, with the latter preferred only when we need to emphasize the dependence on~$a$ and with the former adopted for simplicity of notation. 
This should cause no confusion.

\section{Minimax solutions to the perturbed equations}\label{sect:perturbed solutions}

In this section we consider the perturbed functionals~$J_{\zeta,\eps}$ with~$\zeta\in\bZ\frac{2\pi}{\ell_3}$ and~$\eps\in (0,1]$, and obtain nonzero critical points of~$J_{\zeta,\eps}$ via a minimax method.

We start with the verification of Palais--Smale condition for the perturbed functionals.
The argument goes in the same way as in~\cite{WZ2026Stationary} and we include it for completeness. 
Moreover, we observe that the possible critical levels of~$J_{\zeta,\eps}$ with~$\eps>0$ are necessarily nonnegative. 

\begin{prop}\label{prop:PS}
    Fix~$\eps\in (0,1]$ and let~$(\varphi_n)$ be a~$(PS)_c$ sequence for~$J_{\zeta,\eps}$.
    \begin{itemize}
        \item[(i)] The sequence~$(\varphi_n)$ is bounded in~$H^{\frac{1}{2}}$.
        \item[(ii)] There is a convergent subsequence of~$(\varphi_n)$ whose limit is a solution of~\eqref{eq:perturbed-NDE}.
        \item[(iii)] $c\geq 0$. Moreover, if~$c=0$, then the limit in (ii) is necessarily trivial.
    \end{itemize}
\end{prop}

\begin{proof}
    That~$(\varphi_n)$ is a~$(PS)_c$ sequence means
    \begin{align}\label{eq:PS level}
        J_{\zeta,\eps}(\varphi_n)
        =J_\zeta(\varphi_n)- \eps\int_{\T^2} |\varphi_n|^{\alpha_2}\dv_{\T^2} \to c,
    \end{align}
    and
    \begin{align}\label{eq:PS differential}
        \dd J_{\zeta,\eps}(\varphi_n)=\pD_{\T^2}\varphi_n-a\varphi_n-\partial F_\eps(\varphi_n)\; \to \; 0
         \quad \mbox{ in } H^{-\frac{1}{2}},
    \end{align}
    as~$n\to+\infty$.

    \

    (i) Recall that~$\partial F_\eps(\varphi_n): = \partial F(\varphi_n) + \eps\alpha_2|\varphi_n|^{\alpha_2-2}\varphi_n$. 
    We test~\eqref{eq:PS differential} against~$\varphi_n$ and then subtract it from the twice of~\eqref{eq:PS level}:
    \begin{align}\label{eq: the energy functional minus the F-derivative}
    2J_{\zeta,\eps}(\varphi_n)-\dd J_{\zeta,\eps}(\varphi_n)[\varphi_n]
    =&\int_{\T^2}  \partial F(\varphi_n)[\varphi_n]-2F(\varphi_n)+\eps (\alpha_2-2)|\varphi_n|^{\alpha_2}\dv_{\T^2} \\
    \geq & \eps(\alpha_2-2)\int_{\T^2}  |\varphi_n|^{\alpha_2}\dv_{\T^2},
    \end{align}
    where (F3) is used in the last inequality. Thus
    \begin{align}
    \int_{\T^2}  |\varphi_n|^{\alpha_2}\dv_{\T^2} \leq
    \frac{1}{\eps(\alpha_2-2)} \parenthesis{ 2c + o(1) + o(\|\varphi_n\|_{H^{\frac12}})}.
    \end{align}
    Observe that in the equation
    \begin{align}
    \pD_{\T^2} \varphi_n = a\varphi_n + \partial F(\varphi_n) + \eps \alpha_2 |\varphi_n|^{\alpha_2-2}\varphi_n +\dd J_{\zeta,\eps}(\varphi_n),
    \end{align}
    the right-hand side can be controlled by
    \begin{align}
    \| RHS\|_{H^{-\frac{1}{2}}}
    \leq & a \|\varphi_n\|_{H^{-\frac{1}{2}}} +C(A_2,\alpha_2)( 1+\||\varphi_n|^{\alpha_2-1}\|_{H^{-\frac{1}{2}}}) + o(1) \\
    \leq & C(a,\zeta, A_2, \eps, \alpha_2, \vol(\T^2))(1+ \|\varphi_n\|_{L^{\alpha_2}}^{\alpha_2-1})   + o(1) \\
    \leq&  C(a,\zeta, A_2, \eps, \alpha_2, \vol(\T^2)) \parenthesis{1+\frac{1}{\eps(\alpha_2-2)} \parenthesis{ 2c + o(1) + o(\|\varphi_n\|_{H^{\frac12}})}}^{\frac{\alpha_2-1}{\alpha_2}}.
    \end{align}
    Using the invertibility of~$\pD_{\T^2}$ we see that 
    \begin{align}
    \|\varphi_n\|_{H^{\frac{1}{2}}} \leq C \|\pD_{\T^2}\varphi_n\|_{H^{-\frac{1}{2}}}\leq C(c,a,\zeta, A_2, \eps,\alpha_2,\vol(\T^2)).
    \end{align}

    \

    (ii) By Banach--Alaoglu theorem, passing to a subsequence if necessary, we may assume that~$\varphi_n$ converges  to a limit~$\varphi_\eps\in H^{\frac{1}{2}}$ \emph{weakly} in~$H^{\frac{1}{2}}$ and strongly in~$L^q$ for any~$q<4$.
    In particular,~$\varphi_n$ converges strongly to~$\varphi_\eps$ in~$L^{\alpha_2}$ since~$\alpha_2 <4$. 
    We next show that~$\varphi_n\to\varphi_\eps$ strongly in~$H^{\frac{1}{2}}$ and~$\varphi_\eps$ is a critical point of~$J_{\zeta,\eps}$.

    The continuity assumption (F2) and growth assumption~\eqref{eq:F2-all} then imply 
    \begin{align}
        \p_\varphi F(\cdot, \varphi_n) \to \p_\varphi F(\cdot,\varphi_\eps), & & 
        |\varphi_n|^{\alpha_2-2}\varphi_n \to |\varphi_\eps|^{\alpha_2-2} \varphi_\eps, & & 
        \mbox{ in  } L^{\frac{\alpha_2}{\alpha_2-1}}.
    \end{align}
    Note that~$\frac{\alpha_2}{\alpha_2-1}>\frac{4}{3}$ and~$L^{\frac{4}{3}}(\T^2,\C^4) \hookrightarrow H^{-\frac{1}{2}}(\T^2,\C^4)$.
    Thus we may pass to the limit in~\eqref{eq:PS differential} and obtain that~$\dd J_{\zeta,\eps}(\varphi_\eps)=0$.

    The difference~$\varphi_n-\varphi_\eps$ satisfies
    \begin{align}
        \pD_{\T^2}(\varphi_n-\varphi_\eps)
        =a(\varphi_n-\varphi_\eps)+\partial F(\varphi_n)-\partial F(\varphi_\eps) + \eps \alpha_2 \parenthesis{|\varphi_n|^{\alpha_2-2} \varphi_n - |\varphi_\eps|^{\alpha_2-2}\varphi_\eps }  + \dd J_{\zeta,\eps}(\varphi_n)
    \end{align}
    and the right-hand side above converges to~$0$ in~$H^{-\frac{1}{2}}$. 
    Again by the invertibility of~$\pD_{\T^2}$ we get 
    \begin{align}
        \|\varphi_n-\varphi_\eps\|_{H^{\frac{1}{2}}} \leq C\|\pD_{\T^2}(\varphi_n-\varphi_\eps)\|_{H^{-\frac{1}{2}}} \to 0.
    \end{align}
    Thus~$\varphi_n \to \varphi_\eps$ is actually strong in~$H^{\frac{1}{2}}$. 

    Therefore, the perturbed functional~$J_{\zeta,\eps}$ satisfies the Palais--Smale condition as long as~$\eps\in (0,1]$.

    \

    (iii) Now we have seen that the sequence~$(\varphi_n)$ is bounded in~$H^{\frac{1}{2}}$ and converges in~$H^{\frac{1}{2}}$ to~$\varphi_\eps$. 
    Thus
    \begin{align}
        c=&\lim_{n\to+\infty} J_{\zeta,\eps}(\varphi_n) - \frac{1}{2}\dd J_{\zeta,\eps}(\varphi_n)[\varphi_n] \\
        =&\lim_{n\to+\infty} \int_{\T^2} -F(\varphi_n)+\frac{1}{2}\partial F(\varphi_n)[\varphi_n]\dv_{\T^2}
           +\eps\int_{\T^2} -|\varphi_n|^{\alpha_2} +\frac{1}{2}\alpha_2|\varphi_n|^{\alpha_2}\dv_{\T^2} \\
        \geq & \lim_{n\to+\infty}\frac{1}{2}\eps(\alpha_2-2)\int_{\T^2} |\varphi_n|^{\alpha_2}\dv_{\T^2} \geq 0,
    \end{align}
    because of (F3) and~$\alpha_2>2$.
    Hence all possible critical levels of~$J_{\zeta,\eps}$ are nonnegative.
    Moreover, if~$c=0$, then~$\|\varphi_n\|_{L^{\alpha_2}}\to 0$ and the strong limit is~$\varphi_\eps$. 
    This holds for any $(PS)_0$ sequence, thus the only critical point at the level~$0$ is the zero spinor. 
\end{proof}

Next we search for minimax levels of the perturbed functional.
To this aim we need to find suitable linking structure of the functional~$J_{\zeta,\eps}$.
Although the functional is strongly indefinite, the explicit spectral description of~$\pD_{\T^2}$ provides us with local linking geometry.

Consider the eigenspace of the lowest positive eigenvalue
\begin{align}
        \mathscr{P}= \Span_{\C}\braces{ \begin{pmatrix} 1 \\ 0 \\\frac{\zeta}{ m+\sqrt{\zeta^2+m^2} } \\ 0  \end{pmatrix},
        \begin{pmatrix} 0 \\ 1 \\ 0 \\ \frac{-\zeta}{m+\sqrt{\zeta^2+m^2}} \end{pmatrix} },
\end{align}
which is a vector space of complex dimension two.
It consists of parallel spinors on~$\T^2$, i.e. the component functions are constant.
But unlike the corresponding spinors in~\cite{WZ2026Stationary}, these vectors do not lie in the (+1)-eigenspace of~$\gamma^0$, which is crucial in the construction of links there.
However, we can still see that the upper components occupy at least a fixed portion of the norms of spinors, which is enough for our application.

Given~$\zeta\in\bZ\frac{2\pi}{\ell_3}$, we will use the notation
\begin{align}
    \theta\equiv \theta(\zeta,m)=1-\frac{m^2}{36(\zeta^2+m^2)}\in [\frac{35}{36},1)
\end{align}
throughout this section. 

\begin{lemma}\label{lemma:local control}
    For any~$\mu \in  (\theta\sqrt{\zeta^2+m^2},\sqrt{\zeta^2+m^2})$, and for any~$k\in\braces{1,2,3,4}$, there exists a real~$k$-dimensional space~$\mathscr{E}_{k}\subset H^{\frac{1}{2},+}$ and a positive number~$R>0$ such that for any~$\varphi\in H^{\frac{1}{2}}$, and for any~$\eps\in [0,1]$,
    \begin{multline}
        \parenthesis{ \tnorm{P^-\varphi}_{H^{\frac{1}{2}}} \leq R, \; \mbox{ and }\;  P^+\varphi\in\mathscr{E}_{k}, \, \tnorm{P^+\varphi }_{H^{\frac{1}{2}}}=R }  \\
        \quad \Longrightarrow  \quad
        \int_{\T^2} \frac{1}{2}\Abracket{\pD_{\T^2}\varphi,\varphi}-F_\eps(\varphi) \dv_{\T^2} \leq \int_{\T^2} \frac{\mu}{2}|\varphi|^2 \dv_{\T^2}.
    \end{multline}
\end{lemma}

\begin{proof}
Let~$\mathscr{E}_k$ be a real~$k$-dimensional subspace of~$\mathscr{P}$ consisting of constant spinors.
Then for any~$e\in\mathscr{E}_k$,~$\pD_{\T^2} e=\sqrt{\zeta^2+ m^2} e$.
Moreover,
\begin{align}
    \bar{e}e = \Abracket{\gamma^0 e, e}
    =& |e_{\up}|^2 - |e_{\down}|^2
    = |e_{\up}|^2 - \parenthesis{ \frac{\zeta}{m+\sqrt{\zeta^2+m^2}} }^2 |e_{\up}|^2 \\
    =&\parenthesis{  1 - \frac{\zeta^2}{(m+\sqrt{\zeta^2+m^2})^2} } |e_{\up}|^2
    = \frac{2 m}{m+\sqrt{\zeta^2+m^2}} |e_{\up}|^2,
\end{align}
\begin{align}
    |e|^2=|e_{\up}|^2 +|e_{\down}|^2
    =\parenthesis{  1 + \frac{\zeta^2}{(m+\sqrt{\zeta^2+m^2})^2} } |e_{\up}|^2
    =\frac{2\sqrt{\zeta^2+m^2}}{m+\sqrt{\zeta^2+m^2}}|e_{\up}|^2.
\end{align}
Thus
\begin{align}
    \bar{e}e = \frac{m}{\sqrt{\zeta^2 + m^2}} |e|^2.
\end{align}

Then for~$e=P^+\varphi\in\mathscr{E}_k$ with~$\tnorm{P^+\varphi}_{H^{1/2}} = \tnorm{e}_{H^{1/2}}=R$, we have
\begin{align}
    R^2
    =\int_{\T^2} \Abracket{P^+e,\pD_{\T^2} P^+ e}\dv_{\T^2} = \sqrt{\zeta^2+m^2} \int_{\T^2} |e|^2\dv_{\T^2}
\end{align}
and in particular
\begin{align}
    \int_{\T^2} \frac{1}{2}\Abracket{e,\pD_{\T^2} e}\dv_{\T^2}
    = \int_{\T^2} \frac{\sqrt{\zeta^2+m^2} }{2} |e|^2 \dv_{\T^2}
    = \frac{R^2}{2},
\end{align}
\begin{align}
    \int_{\T^2} \bar{e}e\dv_{\T^2} = \frac{m}{\sqrt{\zeta^2 + m^2}} \int_{\T^2} |e|^2 \dv_{\T^2}
    = \frac{m}{{\zeta^2 + m^2}}\cdot R^2.
\end{align}
For a general~$\varphi=P^-\varphi+e\in H^{\frac{1}{2},-}\oplus\mathscr{E}_k$ with~$\tnorm{P^+\varphi}_{H^{1/2}}=\tnorm{e}_{H^{1/2}}=R$ and~$\tnorm{P^-\varphi}_{H^{1/2}}\leq R$, we consider the following two cases.
\begin{itemize}
    \item If~$-\int_{\T^2}\Abracket{P^-\varphi,\pD_{\T^2} P^-\varphi}\dv_{\T^2}\geq(1-\frac{\mu}{\sqrt{\zeta^2+m^2}})R^2$, then, since~$F_\eps(\varphi)\geq 0$, we have
    \begin{align}
        \int_{\T^2} \frac{1}{2}\Abracket{\pD_{\T^2}\varphi,\varphi}-F_\eps(\varphi) \dv_{\T^2}
        =& \int_{\T^2} \frac{1}{2}\Abracket{P^+\varphi,\pD_{\T^2} P^+\varphi} +\frac{1}{2} \Abracket{P^-\varphi,\pD_{\T^2} P^-\varphi} -F_\eps(\varphi) \dv_{\T^2} \\
         \leq&          \frac{1}{2}\||\pD_{\T^2}|^{\frac12}P^+\varphi\|^2_{L^2} - \frac{1}{2}\||\pD_{\T^2}|^{\frac12}P^-\varphi\|^2_{L^2} \\
        \leq & \frac{1}{2}\cdot R^2 - \frac{1}{2}\cdot(1-\frac{\mu}{\sqrt{\zeta^2+m^2}})R^2\\
        = & \frac{\mu}{2}\cdot\frac{R^2}{\sqrt{\zeta^2+m^2}}  = \frac{\mu}{2}\int_{\T^2} |P^+ \varphi|^2 \dv_{\T^2}  \\
        \leq& \frac{\mu}{2}\int_{\T^2} |\varphi|^2\dv_{\T^2}.
    \end{align}

    \item If, on the other hand, ~$-\int_{\T^2}\Abracket{P^-\varphi,\pD_{\T^2} P^-\varphi}\dv_{\T^2}\leq(1-\frac{\mu}{\sqrt{\zeta^2+m^2}})R^2$, then 
    \begin{align}
    \|P^-\varphi\|_{L^2}^2
    \leq& -\frac{1}{\sqrt{\zeta^2+m^2}} \int_{\T^2} \Abracket{P^-\varphi,\pD_{\T^2} P^-\varphi }\dv_{\T^2} \\
    \leq& \frac{1}{\sqrt{\zeta^2+m^2}}(1-\frac{\mu}{\sqrt{\zeta^2+m^2}})R^2.
    \end{align}
    The negative quadratic contribution alone is insufficient.
    We turn to the nonlinearity to control the positive contribution.
    Indeed, using (F4), we see that
    \begin{align}
    -\int_{\T^2} F_\eps(\varphi)\dv_{\T^2}
    \leq& -\int_{\T^2} A_3 |\bar{\varphi}\varphi|^\nu -A_4 \dv_{\T^2}
    =-A_3\int_{\T^2} |\bar{\varphi}\varphi|^\nu\dv_{\T^2} + A_4 \cdot\vol(\T^2).
    \end{align}
    The term involving the Lorentz scalar~$\bar{\varphi}\varphi$ can be estimated by
    \begin{align}
    &\parenthesis{\int_{\T^2} |\bar{\varphi}\varphi|^\nu\dv_{\T^2} }^{\frac{1}{\nu}}
    \parenthesis{\int_{\T^2} 1 \dv_{\T^2}}^{1-\frac{1}{\nu}}\\
    \geq&  \int_{\T^2} \bar{\varphi}\varphi\dv_{\T^2}
    = \int_{\T^2} \Abracket{\gamma^0\parenthesis{P^+\varphi+ P^-\varphi}, P^+\varphi + P^-\varphi}\dv_{\T^2} \\
    =& \int_{\T^2} \Abracket{\gamma^0 P^+ \varphi, P^+\varphi} + 2\operatorname{Re}\Abracket{ e,\gamma^0P^-\varphi}
    + \Abracket{\gamma^0 P^- \varphi, P^-\varphi}\dv_{\T^2} \\
    \geq & \frac{mR^2}{{\zeta^2+m^2}} -2\|e\|_{L^2}\|P^-\varphi\|_{L^2}-\|P^-\varphi\|_{L^2}^2\\
    \geq&\frac{mR^2}{{\zeta^2+m^2}}-2\parenthesis{\frac{R^2}{\sqrt{\zeta^2+m^2}} }^{\frac{1}{2}}\parenthesis{ \parenthesis{1-\frac{\mu}{\sqrt{\zeta^2+m^2}}}\frac{R^2}{\sqrt{\zeta^2+m^2}}  }^{\frac{1}{2}}  \\
        & -\parenthesis{1-\frac{\mu}{\sqrt{\zeta^2+m^2}}}\frac{R^2}{\sqrt{\zeta^2+m^2}} \\
    =&\frac{R^2}{{\zeta^2+m^2}} \parenthesis{m+\mu-\sqrt{\zeta^2+m^2} -2(\zeta^2+m^2)^{\frac{1}{4}}\sqrt{\sqrt{ \zeta^2+m^2 }  -\mu} }.
\end{align}
To get a positive lower bound of the last line above, we let~$\mu=(1-\delta)\sqrt{\zeta^2+m^2}$ with~$\delta\in(0,1)$ and compute
\begin{align}
   &m+\mu-\sqrt{\zeta^2+m^2} -2(\zeta^2+m^2)^{\frac{1}{4}}\sqrt{\sqrt{ \zeta^2+m^2 }  -\mu}  \\
   =& m+(1-\delta)\sqrt{\zeta^2+m^2}-\sqrt{\zeta^2+m^2} -2(\zeta^2+m^2)^{\frac{1}{4}}\sqrt{\delta\sqrt{\zeta^2+m^2}}\\
   =& m -\delta\sqrt{ \zeta^2+m^2 }-2\sqrt{\delta}\sqrt{ \zeta^2+m^2 }\\
   \geq& m-3\sqrt{\delta}\sqrt{ \zeta^2+m^2 }\\
   \geq& \frac{m}{2} \qquad\qquad \mbox{ provided }\quad \delta\leq\frac{m^2}{36(\zeta^2+m^2)}.
\end{align}
For such~$\mu$, 
\begin{align}
    \int_{\T^2} |\bar{\varphi}\varphi|^\nu \dv_{\T^2}
    \geq  \frac{1}{\vol(\T^2)^{\nu -1}} \parenthesis{\frac{m}{2}\cdot\frac{R^2}{\zeta^2+m^2}}^{\nu}.
\end{align}
In this case we have
\begin{align}
    \int_{\T^2} & \frac{1}{2}\Abracket{\varphi,\pD_{\T^2} \varphi}- F_\eps(\varphi) \dv_{\T^2} \\
    =& \int_{\T^2} \frac{1}{2}\Abracket{P^+\varphi,\pD_{\T^2} P^+\varphi} \dv_{\T^2} + \int_{\T^2}\frac{1}{2}\Abracket{P^-\varphi,\pD_{\T^2}P^-\varphi} \dv_{\T^2} - \int_{\T^2} F_\eps(\varphi)\dv_{\T^2} \\
    =&\frac{1}{2}\parenthesis{ \tnorm{P^+\varphi }_{H^{\frac{1}{2}}}^2 - \tnorm{P^-\varphi}_{H^{\frac{1}{2}}}^2  }
     -\int_{\T^2} F_\eps(\varphi)\dv_{\T^2}\\
    \leq & \frac{1}{2}R^2
    - \frac{A_3m^\nu R^{2\nu}}{\vol(\T^2)^{\nu-1} 2^\nu({\zeta^2+m^2})^\nu}
     +  A_4 \vol(\T^2).
\end{align}
Since~$\nu>1$ and~$A_3>0$, we can choose~$R=R(\vol(\T^2), \zeta,m, A_3, A_4, \nu)\gg 1$ such that the right-hand side above is nonpositive.
For example, it suffices to have 
\begin{align}
    R=\max\braces{\sqrt{2A_4 \vol(\T^2)} , \; \parenthesis{\frac{2^\nu}{A_3}}^{\frac{1}{2(\nu-1)}} \sqrt{\vol(\T^2)} \parenthesis{\frac{\zeta^2+m^2}{m}}^{\frac{\nu}{2(\nu-1)}} }. 
\end{align}
\end{itemize}
The desired conclusion follows.
\end{proof}
Note that the choices of~$\mathscr{E}_k$ and~$R$ are independent of~$\eps\in [0,1]$ and also independent of the frequency~$a\in \R$, a property crucial for later analysis.

\begin{rmk}
    It is possible to get a sharper estimate for the admissible values of~$\mu$.
    Indeed, consider the function~$\eta\colon [0,\sqrt{\zeta^2+m^2}]\to\R$ given by
    \begin{align}
    \eta(\mu)\coloneqq m+\mu-\sqrt{\zeta^2+m^2} -2(\zeta^2+m^2)^{\frac{1}{4}}\sqrt{\sqrt{ \zeta^2+m^2 }  -\mu}.
    \end{align}
    Then~$\eta$ is continuous, increasing and~$\eta(\sqrt{\zeta^2+m^2})=m>0$,~$\eta(\sqrt{\zeta^2+m^2}-m)<0$.
    Thus there exists~$\mu_0> \sqrt{\zeta^2+m^2}-m $ such that~$\eta|_{[\mu_0,\sqrt{\zeta^2+m^2}]}>\frac{m}{2}$.
    Furthermore, the unique root of~$\eta$ is actually given by
    \begin{align}
        \mu_*=\sqrt{\zeta^2+m^2}-\parenthesis{\frac{m}{\sqrt{\sqrt{\zeta^2+m^2}+m}+ (\zeta^2+m^2)^{1/4} } }^2.
    \end{align}
    We take the clean and explicit value of~$\theta$ above because it provides a convenient uniform interval for the linking argument.
    For our purpose in the variational analysis, it suffices to have a nonempty range of~$\mu$ where the conclusion of Lemma~\ref{lemma:local control} holds.

\end{rmk}

For~$k=1$ we take $\mathscr{E}_1=\Span(e)$ with~$\tnorm{e}_{H^{\frac{1}{2}}}=1$ and let
\begin{align}
    \mathscr{C}_1(R)
    \coloneqq \braces{\varphi= P^-\varphi + s e\in H^{1/2}(\T^2,\C^4) \mid  \tnorm{P^-\varphi}_{H^{1/2}}\leq R, \; s \in [0, R]}.
\end{align}
To get a linking structure we consider the candidate
\begin{align}
    \mathscr{S}^+(r)\coloneqq \braces{\varphi\in H^{\frac{1}{2},+} \; \mid \; \tnorm{\varphi}_{H^{1/2}}=r }.
\end{align}
for some~$r\in(0,R)$ to be determined later.
Similar to the situation of~\cite{WZ2026Stationary}, we have
\begin{lemma}\label{lemma:negativity on boundary-1}
    Suppose~$a \in (\theta\sqrt{\zeta^2+m^2},\sqrt{\zeta^2+m^2})$. 
    Then~$J_{\zeta,\eps}|_{\p\mathscr{C}_1(R)}\leq 0$ for any~$\eps\in [0,1]$.
\end{lemma}
\begin{lemma}\label{lemma:positivity in interior-1}
    There exists~$r\in (0,R)$ and~$C_*>0$ such that for any~$\eps\in [0,1]$, we have
    \begin{align}
         J_{\zeta,\eps}|_{\mathscr{S}^+(r)}\geq C_*>0.
    \end{align}
\end{lemma}

\begin{proof}[Proof of Lemma~\ref{lemma:negativity on boundary-1}]
    Let~$\varphi\in\p\mathscr{C}_1(R)$, namely, either~$\tnorm{P^-\varphi}_{H^{1/2}}=R$ or~$s\in \braces{0,R}$.
    \begin{itemize}
        \item If~$\tnorm{P^-\varphi}_{H^{1/2}}=R$, then for any~$s\in[0,R]$, using~$F_\eps(\varphi)\geq 0$,
            \begin{align}
                J_{\zeta,\eps}(P^-\varphi+s e)&\leq \frac{1}{2}(1-\frac{a}{\sqrt{\zeta^2+m^2}})s^2 -  \frac{1}{2}\tnorm{P^-\varphi}^2_{H^{1/2}} \\
                &\leq \frac{1}{2}(1-\frac{a}{\sqrt{\zeta^2+m^2}})R^2 -\frac{1}{2}R^2< 0.
            \end{align}
        \item If~$s=0$, then~$J_{\zeta,\eps}(P^-\varphi)\leq 0$ since~$F_\eps(\varphi)\geq 0$; and if~$s=R$, then~$\tnorm{P^-\varphi}_{H^{1/2}}\leq R$ and~$\tnorm{P^+\varphi}_{H^{1/2}}= R\tnorm{e}_{H^{1/2}}=R$, applying Lemma~\ref{lemma:local control} with~$\mu \in (\theta\sqrt{\zeta^2+m^2}, a]$, we again have~$J_{\zeta,\eps}(\varphi)\leq 0$.
    \end{itemize}
\end{proof}

\begin{proof}[Proof of Lemma~\ref{lemma:positivity in interior-1}]
Note that for~$\varphi=P^+\varphi\in H^{\frac{1}{2},+}$,
\begin{align}
    \|\varphi\|^2_{L^2}\leq \frac{1}{\sqrt{m^2+\zeta^2}}\tnorm{\varphi}^2_{H^{\frac{1}{2}}}.
\end{align}
By (F1) we have 
\begin{align}
    J_{\zeta,\eps}(\varphi)
    \geq &  \frac{1}{2}\int_{\mathbb{T}^2} \Abracket{\varphi, \pD_{\T^2}\varphi} - a|\varphi|^2 \dv_{\T^2}
        - A_1\int_{\mathbb{T}^2} |\varphi|^{\alpha_1} + |\varphi|^{\alpha_2}\dv_{\T^2}\\
    \geq& \frac{1}{2}(1-\frac{a}{\sqrt{m^2+\zeta^2}})\tnorm{\varphi}^2_{H^{\frac{1}{2}}}
    -C(\tnorm{\varphi}^{\alpha_1}_{H^{\frac{1}{2}}}+\tnorm{\varphi}^{\alpha_2}_{H^{\frac{1}{2}}}).
\end{align}
Hence for sufficiently small~$r=\tnorm{\varphi}_{H^{1/2}}>0$, depending only on the fixed data~$m,\zeta,A_1, \vol(\T^2),\alpha_1$ and on~$a$,
\begin{align}
    \inf_{\mathscr{S}^+(r)} J_{\zeta,\eps}(\varphi^+)>0.
\end{align}
\end{proof}

We note that the~$r$ and~$C_*$ can be chosen to be independent of~$\eps\in [0,1]$.
Furthermore,
\begin{align}\label{eq:intersection}
    \mathscr{C}_1(R)\cap \mathscr{S}^+(r)= \braces{ r e} \neq\emptyset.
\end{align}
It follows that
\begin{align}\label{eq:level estimate}
    \sup_{\p \mathscr{C}_1(R)} J_{\zeta,\eps} = 0 
    < C_* 
    \leq  \inf_{\mathscr{S}^+(r)} J_{\zeta,\eps} 
    \leq \sup_{\mathscr{C}_1(R)} J_{\zeta,\eps} <+\infty
\end{align}
for each~$\eps\in [0,1]$, where the last upper bound holds due to the Sobolev embedding together with the fact~$\alpha_2<4$.
That is,~$\p\mathscr{C}_1(R)$ and~$\mathscr{S}^+(r)$ provide the local linking geometry for~$J_{\zeta,\eps}$. 

\

Next we consider deformations of the set~$\mathscr{C}_1(R)$ and try to define the minimax levels in a suitable way. 
Since~$F$ is only assumed to be~$C^{1,\alpha}_{loc}$, the gradient of~$J_{\zeta,\eps}$ is not guaranteed to be locally Lipschitz, hence we cannot directly use the negative gradient flow of~$J_{\zeta,\eps}$. 
The same issue arises in~\cite{WZ2026Stationary} and is resolved there using \emph{admissible pseudo-gradient fields and flows}. 
Those can also be used here on the two-dimensional torus. 
We briefly recall (and adapt suitably) the main constructions before we define the minimax levels.

A vector field~$W$ on~$H^{\frac{1}{2}}(\mathbb{T}^2,\C^4)$ is called an \emph{admissible pseudo-gradient field} if it has the form 
\begin{align}
    W(\varphi)= \eta(\varphi) \parenthesis{ (P^+ -P^- )\varphi - \widetilde{K}(\varphi)},
\end{align}
where
\begin{itemize}
    \item[(W1)] $\eta\colon H^{\frac{1}{2}}\to [0,1]$ is a locally Lipschitz function,
    \item[(W2)] $\widetilde{K}\colon H^{\frac{1}{2}}\to H^{\frac{1}{2}}$ is locally Lipschitz and is compact on bounded sets;
    \item[(W3)] $W|_{\p\mathscr{C}_1(R)} = 0$;
    \item[(W4)] On the region~$\braces{\varphi\colon \eta(\varphi)>0}$, there holds~$\dd J_{\zeta,\eps}(\varphi)[P^+\varphi -P^- \varphi -\widetilde{K}(\varphi)]\geq 0$;
    \item[(W5)] $\tnorm{W(\varphi)}_{H^{1/2}} \leq 1$ for any~$\varphi\in H^{\frac{1}{2}}$. 
\end{itemize}
To such a vector field we associate an admissible pseudo-gradient flow~$\phi^W_t\colon H^{\frac{1}{2}}\to H^{\frac{1}{2}}$ generated by~$-W$: 
\begin{align}
    \begin{cases}
        \frac{\p\phi^W_t(\varphi)}{\p t} = - W(\phi^W_t(\varphi)), & \quad \forall \varphi \in H^{\frac{1}{2}}, \; \forall t\geq 0, \\
        \phi^W_0(\varphi)=\varphi, & \quad \forall \varphi\in H^{\frac{1}{2}}.
    \end{cases}
\end{align}
The flow fixes the boundary~$\p\mathscr{C}_1(R)$ pointwise and does not increase the functional~$J_{\zeta,\eps}$.  
For each~$t\geq 0$, we call~$\phi^W_t$ a time-$t$ map of the admissible pseudo-gradient flow. 
There are plenty of such admissible pseudo-gradient vector fields. 
Moreover, using variation of constants, we see that along the flow the negative part has the form
\begin{align}
    P^-\phi^W_t(\varphi)
    = e^{u^W(t,\varphi)}P^-\varphi + C^W(t,\varphi)
\end{align}
with 
\begin{align}
    u^W(t,\varphi)=\int_0^t \eta(\phi^W_s(\varphi))\dd s, \qquad  0\leq u^W(t,\varphi)\leq t ,
\end{align}
\begin{align}
    C^W(t,\varphi)
    = \int_0^t e^{\int_s^t \eta(\phi^W_\tau(\varphi))\dd \tau } \eta(\phi^W_s(\varphi))P^-\widetilde{K}(\phi^W_s(\varphi))\dd{s}. 
\end{align}
In particular, for any bounded set~$B\subset H^{\frac{1}{2}}$ and any bounded time interval~$[0,T]$, the map 
\begin{align}
    C^W\colon [0,T]\times B \to H^{\frac{1}{2},-}
\end{align}
is continuous and compact.

The deformation class is the collection of finite compositions of time maps of admissible pseudo-gradient flows:
\begin{align}
    \Gamma\coloneqq \braces{ \phi^{W_k}_{t_k}\circ \cdots\circ \phi^{W_1}_{t_1} \mid W_j \mbox{ is admissible }, t_j\geq 0, \; k\in\mathbb{N} }. 
\end{align}
It is nonempty and closed under composition. 
Each element~$\tilde{\phi}$ is homotopic to the identity map by concatenating finitely many flow segments. 
We denote such a homotopy by~$\Phi(t)$:~$\Phi(0)=\id$,~$\Phi(1)=\tilde{\phi}$, and observe that for each~$t\in [0,1]$,
\begin{itemize}
    \item[(1)] $\Phi(t)$ fixes~$\p\mathscr{C}_1(R)$ pointwise and hence~$\Phi(t,\p\mathscr{C}_1(R))\cap \mathscr{S}^+(r) =\emptyset$;
    \item[(2)] $P^-\Phi(t,\varphi)=e^{u(t,\varphi)}P^-\varphi+ C_\Phi(t,\varphi)$ where~$u$ is continuous and bounded, and~$C_\Phi$ is continuous and compact;
    \item[(3)] $\Phi(t)(\mathscr{C}_1(R))$ intersects~$\mathscr{S}^+(r)$:~$\Phi(t,\mathscr{C}_1(R))\cap \mathscr{S}^+(r)\neq \emptyset$, see \cite[Lemma 4.6]{WZ2026Stationary}.
\end{itemize}

We can then define the following minimax level 
\begin{align}
    \Lambda_1(\zeta,\eps) \coloneqq \inf_{\tilde{\phi}\in \Gamma} \sup_{\tilde{\phi}(\mathscr{C}_1(R))} J_{\zeta,\eps}. 
\end{align}
By the standard minimax deformation argument associated with the admissible class~$\Gamma$ (see~\cite{WZ2026Stationary} for details), there exists a Palais--Smale sequence at the level~$\Lambda_1(\zeta,\eps)>0$. 
Indeed, otherwise an admissible pseudogradient deformation would lower the maximal value below~$\Lambda_1(\zeta,\eps)$ contradicting the definition of the minimax level. 
Since the perturbed functional~$J_{\zeta,\eps}$ satisfies the Palais--Smale condition by Proposition~\ref{prop:PS}, this sequence has a convergent subsequence whose limit is a critical point of~$J_{\zeta,\eps}$ at level~$\Lambda_1(\zeta,\eps)$, confirming that~$\Lambda_1(\zeta,\eps)$ is a critical level for~$J_{\zeta,\eps}$. 

We have seen that the critical levels satisfy
\begin{align}
    0<C_*\leq \Lambda_1(\zeta,\eps) \leq C^*<+\infty
\end{align}
with the lower and upper bounds independent of~$\eps\in [0,1]$. 
An explicit upper bound can be given by 
\begin{align}
    \sup_{\mathscr{C}_1(R)} J_{\zeta,\eps} 
    \leq \sup_{ \varphi\in \mathscr{C}_1(R)} \frac{1}{2}\tnorm{P^+\varphi}_{H^{\frac{1}{2}}}^2 - \frac{a}{2}\|P^+\varphi\|_{L^2}^2 = \frac{1}{2}\parenthesis{1-\frac{a}{\sqrt{\zeta^2+m^2}}}R^2. 
\end{align}
Thus we can take, for example,
\begin{align}
    C^*(a)= \frac{1}{2}\parenthesis{1-\frac{a}{\sqrt{\zeta^2+m^2}}}R^2. 
\end{align}
Recall that~$R$ is independent of~$a$ and~$\eps$.
Thus we have 
\begin{align}\label{eq:upper-level-estimate}
    \lim_{a\to \sqrt{\zeta^2+m^2}-0} C^*(a)=0.
\end{align}

Therefore, we obtain
\begin{thm}\label{thm:existence for perturbed NDE}
    For each~$\eps\in (0,1]$ and~$a\in (\theta\sqrt{\zeta^2+m^2},\sqrt{\zeta^2+m^2})$, there exists a nontrivial solution~$\varphi_{\eps,a}$ to~\eqref{eq:perturbed-NDE} such that
    \begin{align}
        0<C_* \leq J_{\zeta,\eps}(\varphi_{\eps,a})=\Lambda_1(\zeta,\eps)\leq C^*(a)<+\infty,
    \end{align}
\end{thm}

We next prove estimates uniform in~$\eps\in (0,1]$ after possibly restricting~$a$ to a smaller left neighborhood of~$\sqrt{\zeta^2+m^2}$. 
Since we aim to get an estimate which is also uniform in~$a$, for~$a$ in a suitable subset of~$(\theta\sqrt{\zeta^2+m^2},\sqrt{\zeta^2+m^2})$, from now on we denote the perturbed functional by~$J_{\zeta,\eps,a}$ to emphasize the dependence on~$a$. 

\begin{thm}\label{thm:uniform estimate}
    There exists~$m_*\in (\theta\sqrt{\zeta^2+m^2},\sqrt{\zeta^2+m^2})$ and~$C>0$, depending on the structure constants~$A_j$,~$\alpha_1,\alpha_2,\nu,\beta$, the physical parameters~$m,\zeta$ and the torus~$\T^2$, such that for any~$a\in (m_*, \sqrt{\zeta^2+m^2})$ and any~$\eps\in (0,1]$, 
    \begin{align}\label{eq:uniform estimate}
         0<\|\varphi_{\eps,a}\|_{H^1(\T^2,\C^4)}\leq C. 
    \end{align}
\end{thm}

\begin{proof}
    The proof follows the strategy of~\cite{WZ2026Stationary}, with suitable modifications reflecting the shifted lowest eigenspaces. 
    We will first bound the nonlinearity part in integral sense, then use a contradiction argument to show that the perturbed solutions have uniformly bounded~$L^4$ norms, and then use a bootstrap argument to get a uniform bound in~$H^1$. 

    \noindent\textbf{Step 1.} \emph{There exists a constant~$C=C(\alpha_1, C^*(a))>0$ such that
    \begin{align}
        0\leq \int_{\T^2} F_\eps(\varphi_{\eps,a})\, \dv_{\T^2}\leq C.
    \end{align}
    }

    We use~\eqref{eq:perturbed-NDE} and (F3) to get
    \begin{align}
        2J_{\zeta,\eps,a}(\varphi_{\eps,a}) + 2\int_{\T^2} F_\eps(\varphi_{\eps,a})\dv_{\T^2}
        =&\int_{\T^2} \Abracket{\varphi_{\eps,a},\pD_{\T^2}\varphi_{\eps,a}}-a|\varphi_{\eps,a}|^2\dv_{\T^2} \\
        =&\int_{\T^2} \partial F_\eps(\varphi_{\eps,a})[\varphi_{\eps,a}]\dv_{\T^2} \\
        \geq& \alpha_1\int_{\T^2} F_\eps(\varphi_{\eps,a}) \dv_{\T^2}.
    \end{align}
    As~$\alpha_1>2$ and~$J_{\zeta,\eps}(\varphi_{\eps,a})\in [C_*, C^*(a)]$, it follows that
    \begin{align}
        0\leq \int_{\T^2} F_\eps(\varphi_{\eps,a})\dv_{\T^2} \leq \frac{2C^*(a)}{\alpha_1-2}.
    \end{align}
    Since~$F_\eps$ is a sum of two nonnegative summands, we have 
    \begin{align}\label{eq:uniform estimate for perturbation}
    0\leq \int_{\T^2} F(\varphi_{\eps,a})\dv_{\T^2}\leq \frac{2C^*(a)}{\alpha_1-2}, & & 
    \mbox{ and } & & 
    0\leq \eps\int_{\T^2} |\varphi_{\eps,a}|^{\alpha_2}\dv_{\T^2}\leq \frac{2C^*(a)}{\alpha_1-2}.
    \end{align}
    Note that the above bound on the~$L^{\alpha_2}$ norms blows up as~$\eps\to0^+$.

    \

\noindent\textbf{Step 2.} \emph{We establish a uniform bound of the norms~$\|\varphi_\eps\|_{L^4}$.} 

Argue by contradiction, and suppose that there is no such uniform~$L^4$ bound on the perturbed solutions, no matter how close~$a$ is to the lowest positive eigenvalue~$\sqrt{\zeta^2+m^2}$. 
Then there would exist sequences~$(a_n)\subset (\theta\sqrt{\zeta^2+m^2},\sqrt{\zeta^2+m^2})$ and~$(\eps_n)\subset (0,1]$ such that
\begin{align}
    a_n\to\sqrt{\zeta^2+m^2}
    & & \mbox{ and } & &
    \lambda_n=\|\varphi_{\eps_n,a_n}\|_{L^4(\T^2)}\to+\infty. 
\end{align}
Passing to a subsequence if necessary, the sequence~$\eps_n$ also converges, say, to~$\eps_*\in [0,1]$. 
Note that 
\begin{align}\label{eq:NDE-an}
    (\pD_{\T^2} -a_n) \varphi_{\eps_n,a_n} 
    = \p F(\varphi_{\eps_n,a_n}) 
    + \alpha_2 \eps_n |\varphi_{\eps_n,a_n}|^{\alpha_2 -2} \varphi_{\eps_n,a_n}, 
\end{align}
\begin{align}
    0<J_{\zeta,\eps_n,a_n} (\varphi_{\eps_n,a_n})  \leq C^*(a_n) \to 0, \quad \mbox{ as } a_n \to \sqrt{\zeta^2+m^2}. 
\end{align}

\ 

\emph{Claim: $\eps_*=0$.}
\begin{proof}[Argument for Claim.] If~$\eps_*>0$, then we may assume that~$\eps_n>\frac{\eps_*}{2}>0$ for all~$n\geq 1$. 
    From~\eqref{eq:uniform estimate for perturbation} it follows that
    \begin{align}
        \|\varphi_{\eps_n,a_n}\|_{L^{\alpha_2}}^{\alpha_2} \leq \frac{4C^*(a_n)}{\eps_*(\alpha_1 -2)} \to 0.
    \end{align}
    By~\eqref{eq:NDE-an} and using~\eqref{eq:F2-all}, we get 
    \begin{align}
        \|\pD_{\T^2} \varphi_{\eps_n,a_n}\|_{L^{\frac{\alpha_2}{\alpha_2 -1}}(\T^2)}
        =& \| a_n \varphi_{\eps_n,a_n} +\p F(\varphi_{\eps_n,a_n}) 
    + \alpha_2 \eps_n |\varphi_{\eps_n,a_n}|^{\alpha_2 -2} \varphi_{\eps_n,a_n} \|_{L^{\frac{\alpha_2}{\alpha_2-1}}} \leq C<+\infty. 
    \end{align}
    Since~$\frac{\alpha_2}{\alpha_2 -1}>\frac{4}{3}$, and~$W^{1,\frac{4}{3}}(\T^2)\hookrightarrow L^4(\T^2)$, this gives a uniform bound of~$\|\varphi_{\eps_n,a_n}\|_{L^4}$, contradicting that~$\lambda_n\to+\infty$.
    Thus~$\eps_* =\lim\limits_{n\to+\infty} \eps_n =0$. 
\end{proof}

Normalize the solutions~$\varphi_{\eps_n,a_n}$ to
\begin{align}
    \vartheta_n\coloneqq \frac{\varphi_{\eps_n,a_n}}{\lambda_n}, & & \|\vartheta_n\|_{L^4}=1, & & n\geq 1,
\end{align}
which satisfy the equations
\begin{align}\label{eq:normalized eqn}
    \pD_{\T^2}\vartheta_n-a_n\vartheta_n=\frac{\p F(\varphi_{\eps_n,a_n})}{\lambda_n} + \alpha_2\eps_n |\varphi_{\eps_n,a_n}|^{\alpha_2 -2 }\vartheta_n, \qquad \mbox{ on } \; \mathbb{T}^2.
\end{align}
The perturbation terms converge to zero in~$L^{\frac{\alpha_2}{\alpha_2-1}}$:
\begin{align}
    \|\alpha_2\eps_n |\varphi_{\eps_n,a_n}|^{\alpha_2 -2 }\vartheta_n \|_{L^{\frac{\alpha_2}{\alpha_2 -1}}}
    = & \frac{\alpha_2 \eps_n^{\frac{1}{\alpha_2}} }{\lambda_n}
     \parenthesis{ \int_{\mathbb{T}^2} \left|\eps_n^{\frac{\alpha_2-1}{\alpha_2}} |\varphi_{\eps_n,a_n}|^{\alpha_2 -1} \right|^{\frac{\alpha_2}{\alpha_2-1}} \dv_{\T^2}   }^{\frac{\alpha_2-1}{\alpha_2}}\\
    \leq& \frac{\alpha_2 \eps_n^{\frac{1}{\alpha_2}} }{\lambda_n}
    \parenthesis{ \int_{\mathbb{T}^2} \eps_n |\varphi_{\eps_n,a_n}|^{\alpha_2}\dv_{\T^2}}^{\frac{\alpha_2-1}{\alpha_2}} \\
    \leq& \frac{\alpha_2 \eps_n^{\frac{1}{\alpha_2}} }{\lambda_n}\parenthesis{ \frac{2C^*(a_n)}{\alpha_1-2} }^{\frac{\alpha_2-1}{\alpha_2}} \to 0 ,\qquad \mbox{ as } n\to+\infty. 
\end{align}
For the nonlinear term involving~$\p F$, we need a more careful analysis.
The assumption (F5) implies the pointwise estimate (with~$\delta=1$)
\begin{align}
    \left|\frac{1}{\lambda_n} \p F(\varphi_{\eps_n,a_n})\right|
    \leq C \parenthesis{1+ F(\varphi_{\eps_n,a_n})^{\frac{1}{\beta}}} |\vartheta_n| \in L^{\frac{4\beta}{4+\beta}},
\end{align}
hence the norm estimate
\begin{align}
    \|\frac{1}{\lambda_n}\p F(\varphi_{\eps_n,a_n})\|_{L^{\frac{4\beta}{4+\beta}}}
    \leq C \parenthesis{ \vol(\mathbb{T}^2)+\frac{2 C^*(a_n)}{\alpha_1-2} }^{\frac{1}{\beta}}
\end{align}
due to \noindent\textbf{Step 1.}
Let 
\begin{align}
    p\coloneqq \min\braces{ \frac{4\beta}{4+\beta}, \frac{\alpha_2}{\alpha_2-1} } >\frac{4}{3}. 
\end{align}
Then~$ (\pD_{\T^2}-a_n)\vartheta_n \in L^p(\mathbb{T}^2)$ with
\begin{align}
    \|(\pD_{\T^2}-a_n)\vartheta_n\|_{L^p} \leq 2C \parenthesis{ \vol(\mathbb{T}^2)+\frac{2 C^*(a_n)}{\alpha_1-2} }^{\frac{1}{\beta}}.
\end{align}
It then follows that the sequence ~$(\vartheta_n)$ is uniformly bounded in~$W^{1,p}$ and,  upon passing to a subsequence, converges in~$L^4$ to some~$\vartheta_\infty$ with~$\|\vartheta_\infty\|_{L^4}=1$.
Moreover, by (F4) and~\eqref{eq:uniform estimate for perturbation}, 
\begin{align}
    \int_{\mathbb{T}^2} |\bar{\vartheta}_{n}\vartheta_{n}|^\nu\dv_{\T^2} \leq \frac{1}{A_3\lambda_n^{2\nu}} \parenthesis{ \int_{\mathbb{T}^2} F_{\eps_n}(\varphi_{\eps_n,a_n})\dv_{\T^2} + A_4\vol(\mathbb{T}^2)}
    \leq \frac{C+A_4 \vol(\mathbb{T}^2)}{A_3\lambda_n^{2\nu}}\to 0
\end{align}
as~$n\to+\infty$.
It follows that~$\bar{\vartheta}_\infty\vartheta_\infty\equiv 0$ almost everywhere. 
That is,~$\vartheta_\infty$ lies in the Lorentz null cone~$\mathcal{N}$,
so it cannot be in~$\Eigen(\pD_{\T^2};\sqrt{\zeta^2+m^2})$.

Recall that~$\proj_{m,\zeta}$ is the~$L^2$-orthogonal projection to~$\Eigen(\pD_{\T^2};\sqrt{m^2+\zeta^2})$ and~$\proj_{m,\zeta}^\perp = \id -\proj_{m,\zeta}$ is the complementary projection.

\ 

Now we have~$\|\bar{\vartheta}_n\vartheta_n\|_{L^\nu}\to 0$ and~$\|\vartheta_n\|_{L^4}=1$. 
By Lemma~\ref{lemma:null-cone-separation},
\begin{align}
    \|\proj_{m,\zeta}^\perp\vartheta_n\|_{L^4}\geq \kappa_2>0. 
\end{align}
On the other hand, using~\eqref{eq:normalized eqn}, (F5) and~\eqref{eq:uniform estimate for perturbation}, we have  
\begin{align}
    \| (\pD_{\T^2}-a_n)\vartheta_n\|_{L^{\frac{4}{3}}} 
    \leq& C(m,\zeta,\vol(\mathbb{T}^2),\beta,\alpha_2)\parenthesis{ \|\frac{1}{\lambda_n}\p F(\varphi_{\eps_n,a_n})\|_{L^{\frac{4\beta}{4+\beta}} } + \alpha_2\|\eps_n|\varphi_{\eps_n,a_n}|^{\alpha_2-2}\vartheta_n \|_{L^{\frac{\alpha_2}{\alpha_2-1}} } } \\
    \leq & C(m,\zeta,\vol(\mathbb{T}^2),\beta,\alpha_2)
    \parenthesis{ A_5\parenthesis{\delta+ C_\delta \parenthesis{ \int_{\mathbb{T}^2} F(\varphi_{\eps_n,a_n})\dv_{\T^2}}^{\frac{1}{\beta}} } + \frac{\eps_n^{\frac{1}{\alpha_2}}}{\lambda_n} } \\
    \leq &  C(m,\zeta,\vol(\mathbb{T}^2),\beta,\alpha_2, A_5)
    \parenthesis{ \delta + C_\delta (C^*(a_n))^{\frac{1}{\beta} } + \frac{\eps_n^{\frac{1}{\alpha_2}}}{\lambda_n} }.
\end{align}
Combining these two estimates using Lemma~\ref{lemma:invertibility-orthogonal} leads to 
\begin{align}
   0<\kappa_2\leq & \|\proj_{m,\zeta}^\perp \vartheta_n\|_{L^4}
    \leq C(m,\zeta) \|\proj_{m,\zeta}^\perp (\pD_{\T^2}-a_n)\vartheta_n\|_{L^{\frac{4}{3}}}  \\
    \leq & C(m,\zeta,\vol(\mathbb{T}^2),\beta,\alpha_2, A_5)
    \parenthesis{ \delta + C_\delta (C^*(a_n))^{\frac{1}{\beta} } + \frac{\eps_n^{\frac{1}{\alpha_2}}}{\lambda_n} }.
\end{align}
By first choosing~$\delta>0$ small, then sending~$n\to+\infty$ and recalling~$C^*(a_n)\to 0^+$ and~$\lambda_n\to+\infty$, we get a contradiction. 

This confirms the existence of~$m_*\in (\theta\sqrt{\zeta^2+m^2},\sqrt{\zeta^2+m^2})$ and~$C>0$ such that for any~$\eps\in (0,1]$ and any~$a\in (m_*,\sqrt{\zeta^2+m^2})$, we have 
\begin{align}
    \|\varphi_{\eps,a}\|_{L^4}\leq C. 
\end{align} 

\ 

\noindent\textbf{Step 3.} \emph{The perturbed solutions~$\varphi_{\eps,a}$ obtained as above are also uniformly bounded in~$H^1(\T^2;\C^4)$.}

With the above uniform~$L^4$ bound on~$\varphi_{\eps,a}$, we see that 
    \begin{align}
        |\pD_{\T^2}\varphi_{\eps,a}|
        =|a\varphi_{\eps,a} + \p F(\varphi_{\eps,a})+ \alpha_2\eps|\varphi_{\eps,a}|^{\alpha_2-2}\varphi_{\eps,a} |
        \leq
        C\parenthesis{ 1+ |\varphi_{\eps,a}|^{\alpha_2 -1}}
    \end{align}
    which is uniformly bounded in~$L^{\frac{4}{\alpha_2 -1 }}$ with~$\frac{4}{\alpha_2 -1}>\frac{4}{3}$.

    If~$\alpha_2\in (2,3]$, then~$\frac{4}{\alpha_2-1}\geq 2$, so the right-hand side of the Dirac equation is already uniformly bounded in~$L^2$. 
    Elliptic regularity immediately gives
    \begin{align}
        \|\varphi_{\eps,a}\|_{H^1}\leq C<+\infty, \quad \forall \eps\in (0,1].
    \end{align}
    If~$\alpha_2\in (3,4)$, then~$\frac{4}{\alpha_2-1}\in (\frac{4}{3},2)$, so the spinors~$\varphi_{\eps,a} $ are uniformly bounded in~$W^{1,\frac{4}{\alpha_2-1}} \hookrightarrow L^{\frac{4}{\alpha_2-3}}$ and~$\frac{4}{\alpha_2-3}>4$.
    We repeat this subcritical bootstrap procedure finitely many times until the exponent is sufficiently large so that the right-hand side belongs to~$L^2$. This ultimately yields a uniform~$H^1$-bound.

\end{proof}

\begin{proof}[Proof of Theorem~\ref{thm:existence-splitting}]
    Fix~$a\in (m_*,\sqrt{\zeta^2+m^2})$ and consider the minimax critical points~$\varphi_{\eps,a}$ obtained in Theorem~\ref{thm:existence for perturbed NDE}, which are uniformly bounded due to Theorem~\ref{thm:uniform estimate}. 
    Thus we can take a sequence~$\eps_n\to0^+$ for which~$\varphi_{\eps_n,a}$ converges to~$\varphi_{0,a}$ weakly in~$H^1$ and strongly in~$H^{\frac{1}{2}}$, and also strongly in~$L^q$ for every finite~$q\in [1,+\infty)$. 
    The perturbation terms go to zero in the limit both in the functional~\eqref{eq:perturbed-functional} and in the equation~\eqref{eq:perturbed-NDE}; for example,  
    \begin{align}
        \|\eps_n |\varphi_{\eps_n,a}|^{\alpha_2-2}\varphi_{\eps_n,a}\|_{L^2}
        \leq \eps_n \|\varphi_{\eps_n,a}\|_{L^{2(\alpha_2-1)}}^{\alpha_2 -1}
        \leq C\eps_n\|\varphi_{\eps_n,a}\|_{H^1}^{\alpha_2 -1} \leq C\eps_n\to 0, \quad \mbox{ as } \eps_n\to 0^+.        
    \end{align}
    Therefore~$\varphi_{0,a}$ is a solution of~\eqref{eq:NDE-T2}, with
    \begin{align}
        J_{\zeta}(\varphi_{0,a}) = \lim_{\eps_n\to 0^+} J_{\zeta,\eps_n}(\varphi_{\eps_n,a}) \geq C_*>0.
    \end{align}
    Hence also~$\varphi_{0,a}\neq 0$.
    The growth estimate~\eqref{eq:F2-all} first yields~$\varphi_{0,a}\in W^{1,p}$ for some~$p>2$, hence~$\varphi_{0,a}\in C^\sigma$ for some~$\sigma\in (0,1)$. 
    Since~$\p_\varphi F\in C^{\alpha}_{loc}$, Schauder theory applied to~\eqref{eq:NDE-T2} then gives~$\varphi_{0,a}\in C^1$. 
    This gives a nonzero~$C^1$ solution~$\psi_\zeta=\varphi_{0,a}\otimes e^{i\zeta x^3}$ to~\eqref{eq:NDE-spatial}. 

    This machinery works for each~$\zeta\in\bZ\frac{2\pi}{\ell_3}$, as long as the frequency~$a$ lies in a small left neighborhood of~$\sqrt{\zeta^2+m^2}$. This finishes the proof.   
\end{proof}

\ 

\begin{center}
    {\scriptsize DECLARATIONS}
\end{center} 

\

\textbf{Conflict of Interest.} The authors have no conflict of interest.

\textbf{Data Availability Statement.} Data sharing is not applicable because no datasets were generated or analyzed in this study.

\

\bibliographystyle{plainurl}
\bibliography{nonlinearDiracEqn}

\end{document}